\documentclass[11pt]{article}
\usepackage{amsthm}
\usepackage{amsmath,amssymb,  amscd,eucal}
 \usepackage{latexsym}
\usepackage{graphicx}
\usepackage[usenames]{color}
\usepackage{times,tikz}
\usepackage{paralist}
\usetikzlibrary{decorations.markings}

\usepackage[letterpaper,margin=1.00in]{geometry}

\usetikzlibrary{shapes.geometric,shapes.arrows,decorations.pathmorphing}
\usetikzlibrary{matrix,chains,scopes,positioning,arrows,fit}
\tikzset{
  treenode/.style = {align=center, inner sep=0pt, text centered,
    font=\sffamily},
  arn_n/.style = {treenode, circle, white, font=\sffamily\bfseries, draw=black,
    fill=black, text width=1.5em},% arbre rouge noir, noeud noir
  arn_r/.style = {treenode, circle, red, draw=red,
    text width=1.5em, very thick},% arbre rouge noir, noeud rouge
  arn_x/.style = {treenode, rectangle, draw=black,
    minimum width=0.5em, minimum height=0.5em}% arbre rouge noir, nil
}

\newtheorem{thm}[equation]{Theorem}
\newtheorem{cor}[equation]{Corollary}

\newtheorem{prop}[equation]{Proposition}

\newtheorem{remark}[equation]{Remark}
\newtheorem{example}[equation]{Example}

\numberwithin{equation}{section}
\newcommand{\angstrom}{\mbox{\normalfont\AA}}

\renewcommand{\det}{\mathsf{det}}

\newcommand{\overbar}[1]{\mkern 1.2mu\overline{\mkern-1mu#1\mkern-1mu}\mkern 1mu}
\newcommand{\FF}{\mathbb{F}}
\newcommand{\ZZ}{\mathbb{Z}}
\newcommand{\CC}{\mathbb{C}}
\newcommand{\GL}{\mathrm{GL}}
\newcommand{\SL}{\mathrm{SL}}

\newcommand{\cR}{{\mathsf{Q}}}

\newcommand{\arm}{\mathsf{a}}
\newcommand{\bal}{\underline{\alpha}}
\newcommand{\br}{\mathsf{b}}
\newcommand{\bs}{\boldsymbol }
\newcommand{\cm}{\mathsf{c}}

\newcommand{\er}{\mathrm{e}}
\newcommand{\gr}{\mathrm{g}}
\newcommand{\hr}{\mathsf{h}}

\newcommand{\Sr}{\mathrm{S}}

\newcommand{\lam}{\lambda}
\newcommand{\ve}{\varepsilon}
\newcommand{\es}{\mathsf{e}}
\newcommand{\Fs}{\mathsf{F}}

\newcommand{\ms}{\mathsf{m}}
\newcommand{\om}{\omega}
\newcommand{\Ps}{\mathsf{P}}

\newcommand{\ws}{\mathsf{w}}
\newcommand{\EE}{\mathsf{E}}
\newcommand{\GG}{\mathsf{G}}
\newcommand{\BB}{\mathsf{B}}
\newcommand{\col}{\mathsf{M}}

\newcommand{\TT}{\mathbf{T}}
\newcommand{\II}{\mathbf{I}}
\newcommand{\OO}{\mathbf{O}}

\newcommand{\DD}{\mathbf{D}}
\newcommand{\UU}{\mathsf{U}}
\newcommand{\VV}{\mathsf{V}}

 \newcommand{\SU}{\mathrm{SU}}
\newcommand{\Ss}{\mathsf{S}}
\newcommand{\zr}{\mathsf{0}}
\newcommand{\legendre}[2]{\genfrac{(}{)}{}{}{#1}{#2}}

\newcommand{\A}{\mathsf{A}}

\newcommand{\Zs}{\mathsf{Z}}

\newcommand\End{\mathsf {End}}

\newcommand\half{\frac{1}{2}}

\renewcommand{\cosh}{\mathsf{cosh}}
\renewcommand{\sinh}{\mathsf{sinh}}
\newcommand{\onto}{\mathrel{\rightarrow\hspace{-1.5ex}\to}}

\newcommand{\binombr}[2]{\genfrac{\{}{\}}{0pt}{}{#1}{#2} }

\def \ot {\otimes}
\def\modd{\, \mathsf{mod} \,}

\def\dimm{\, \mathsf{dim}\,}

\begin{document}
\title{\qquad \  Walks on Graphs and Their Connections with\newline
Tensor Invariants and Centralizer Algebras}
\author{Georgia Benkart and Dongho Moon\thanks{This research was supported by the Basic Science Research Program of the National Research Foundation of Korea (NRF) funded by the Ministry of Education (2015R1D1A1A01057484).
The hospitality of the Mathematics Department at the University of Wisconsin-Madison while this research
was done  is gratefully
acknowledged.}}
\date{}
%10-31-16
\maketitle
\begin{abstract} The number of walks of $k$ steps from the node $\zr$ to the  node $\lam$ on the McKay quiver determined
by a finite group $\GG$ and a $\GG$-module $\VV$ is the multiplicity of the irreducible $\GG$-module $\GG_\lam$ in the tensor power $\VV^{\ot k}$,  and it is also the dimension of the irreducible module labeled by $\lam$ for the centralizer algebra $\Zs_k(\GG) = \End_\GG(\VV^{\ot k})$.  This paper explores ways to effectively calculate that number using the character theory of $\GG$.
We determine the corresponding Poincar\'e series.  The special case  $\lam = \zr$ gives  the Poincar\'e series for
the tensor invariants $\mathsf{T}(\VV)^\GG = \bigoplus_{k =0}^\infty (\VV^{\ot k})^\GG$ and a tensor analog of
Molien's formula for polynomial invariants.   When $\GG$ is abelian,  we show that the exponential generating
function for the number of walks is a product of generalized hyperbolic functions.  Many graphs (such
as circulant graphs) can be viewed as McKay quivers,  and the methods presented here
provide efficient ways to compute the number of walks on them.    \end{abstract}

\textbf{MSC Numbers (2010)}: \, 05E10, 20C05


\section{Introduction}
Let  $\GG$ be a finite group,  and assume that the elements $\lam$ of $ \Lambda(\GG)$ index  the irreducible complex representations
of $\GG$, hence also the conjugacy classes of $\GG$.
Let $\GG_\lam$ denote the irreducible $\GG$-module indexed by $\lam$, and let $\chi_\lam$ be its character.
The module $\GG_\zr$ denotes  the trivial one-dimensional $\GG$-module with $\chi_\zr(g) = 1$ for all $g \in \GG$.

The  \emph{McKay quiver} $\cR_\VV(\GG)$ (also known as the \emph {representation graph})
associated to a finite-dimensional $\GG$-module $\VV$ over the complex field $\CC$ has nodes corresponding to the irreducible $\GG$-modules $\{\GG_\lambda \mid \lam \in \Lambda(\GG)\}$.
For $\nu  \in \Lambda(\GG)$, there are $a_{\nu,\lam}$  arrows
from $\nu$ to $\lam$ in  $\cR_\VV(\GG)$  if
\begin{equation}\label{eq:gtens} \GG_\nu \ot \VV =  \bigoplus_{\lam \in \Lambda(\GG)} a_{\nu,\lam}  \GG_\lam.\end{equation}
If $a_{\nu,\lam} = a_{\lam,\nu}$,  then we draw $a_{\nu,\lam}$ edges without arrows between $\nu$ and $\lam$.
 The number of arrows $a_{\nu,\lam}$ from $\nu$ to  $\lam$ in  $\cR_\VV(\GG)$ is the multiplicity of $\GG_\lambda$  as a summand of $\GG_\nu \ot \VV$.
Since each step on the graph is achieved by tensoring with $\VV$,
\begin{align}\begin{split}\label{eq:multwalk} \ms_k ^\lam: &= \text{number  of walks of $k$ steps from $\zr$  to} \ \lam \\
& = \text{multiplicity of  $\GG_\lam$  in  \, $\GG_{\zr} \ot  \VV^{\ot k}  \cong \VV^{\ot k}$.}
\end{split}\end{align}

For a faithful $\GG$-module $\VV$,
 any  irreducible $\GG$-module $\GG_\lambda$ occurs in $\VV^{\otimes \ell}$ for some $\ell$ by Burnside's theorem (in fact, for some $\ell$ such that
$0 \leq \ell \leq |\GG|$ by Brauer's strengthening of that result  \cite[Thm. 9.34]{CR}).
  This implies that there is a directed path with $\ell$ steps from $\GG_{\zr}$ to $\GG_\lambda$ in
 $\cR_{\VV}(\GG)$.

The {\it centralizer algebra},
\begin{equation}\label{eq:cent1} \Zs_k(\GG) = \{z \in \End(\VV^{\ot k}) \mid  z(g.w) = g.z(w) \  \ \forall \ g \in \GG,  w \in \VV^{\ot k}\}, \end{equation}
plays a critical role in studying $\VV^{\ot k}$, as it contains the projection maps onto  the irreducible summands of $\VV^{\ot k}$.

Let $\Lambda_k(\GG)$
denote the subset of $\Lambda(\GG)$ corresponding to the  irreducible $\GG$-modules that occur in $\VV^{\ot k}$ with
multiplicity at least one.    \emph{Schur-Weyl duality}
establishes essential connections between the representation theories of $\GG$ and $\Zs_k(\GG)$:
\begin{itemize}
\smallskip
\item[$\bullet$] $\Zs_k(\GG)$ is a semisimple associative $\CC$-algebra whose  irreducible modules $\Zs_k^\lam(\GG)$
are in bijection with the elements $\lambda$ of $\Lambda_k(\GG)$.
\smallskip
\item[$\bullet$] $\dimm \Zs_k^\lam(\GG)  = \ms_k^\lambda$, the number of walks of $k$ steps from
the trivial $\GG$-module $\GG_\zr$ to $\GG_\lambda$ on $\cR_{\VV}(\GG)$.  \smallskip
\item[$\bullet$] If $d_\lam = \dimm \GG_\lambda$, then the tensor space $\VV^{\otimes k}$ has the following decompositions:  \begin{equation}
\begin{array}{rll}
\VV^{\otimes k} & \cong \displaystyle{\bigoplus_{\lambda\in \Lambda_k(\GG)}}\,\, \ms_k^\lambda\, \GG_\lambda & \hbox{ as a $\GG$-module}, \\
& \cong \displaystyle{\bigoplus_{\lambda \in \Lambda_k(\GG)}}\,  d_\lambda \, \Zs_k^\lam(\GG) & \hbox{ as a $\Zs_k(\GG)$-module}, \\
& \cong \displaystyle{\bigoplus_{\lambda \in \Lambda_k(\GG)}}\left(\GG_\lambda \otimes \Zs_k^\lam(\GG) \right) & \hbox{ as a $(\GG,\Zs_k(\GG))$-bimodule}. \\
\end{array} \end{equation}
\end{itemize}

In Corollary \ref{C:centdim} below,  we contribute three more important relations to this list of Schur-Weyl duality results:
\begin{itemize}
\item[$\bullet$]  $\dimm  \Zs_k^\lam(\GG) = \vert \GG \vert^{-1} \displaystyle{\sum_{g \in \GG}  \ \chi_{{}_{\VV}}(g)^k \
 \overbar{\chi_{\lam}(g)},}$
 \item[$\bullet$]  $\dimm (\VV^{\ot k})^\GG = \vert \GG \vert^{-1} \displaystyle{ \sum_{g \in \GG}}  \ \chi_{{}_{\VV}}(g)^{k},$
 where  \\
  $(\VV^{\ot k})^\GG = \{w \in \VV^{\ot k} \mid  g.w = w \, \ \forall g \in \GG\}$ \quad  (the space of $\GG$-invariants in
 $\VV^{\ot k}$),
\item[$\bullet$]  $\dimm  \Zs_k(\GG)  = \vert \GG \vert^{-1} \displaystyle{ \sum_{g \in \GG}}  \ \chi_{{}_{\VV}}(g)^{2k},$  \quad
when $\VV$ is a self-dual $\GG$-module,
\end{itemize}
where  $ \chi_{{}_{\VV}}$ is the character of $\VV$, and $\chi_\lam$ is the character of the irreducible $\GG$-module $\GG_\lam$.

Therefore, Schur-Weyl duality tells us that the following numbers are the same,  and our aim in this paper is to demonstrate various ways to compute these values
effectively:
\begin{itemize}
\item[{\rm (1)}]   the number of walks of $k$ steps  from $\zr$ to $\lambda \in \Lambda(\GG)$ on $\cR_{\VV}(\GG)$,
\item[{\rm (2)}]  the $(\zr, \lam)$-entry $(\A^k)_{{\zr}, \lambda}$ of $\A^k$,  where $\A=(a_{\nu,\lam})$ is the adjacency matrix of  $\cR_{\VV}(\GG)$,
\item[{\rm (3)}]   the multiplicity $\ms_k^\lam$  of the irreducible $\GG$-module $\GG_\lambda$ in $\VV^{\ot k}$,
\item[{\rm (4)}]  the dimension of the irreducible module $\Zs_k^\lam(\GG)$  labeled by $\lambda \in \Lambda_k(\GG)$ for the centralizer algebra $\Zs_k(\GG) = \End_{\GG}(\VV^{\ot k})$,
\item[{\rm (5)}]  the number of paths from $\zr$ at level 0 to $\lambda$ at level $k$ on the Bratteli diagram $\mathcal B_{\VV}(\GG)$ (see Section \ref{S:Brat}  for the definition).\smallskip
\item[{\rm ($\ast$)}] Moreover, when $\lam = \zr$,  these values are all equal to the dimension  $\dimm(\VV^{\ot k})^\GG$
 of the space of $\GG$-invariants.  \end{itemize}

Many graphs can be viewed as McKay quivers $\cR_{\VV}(\GG)$ for some choice of $\GG$ and $\VV$,
and the methods described here provide an efficient approach to computing walks on them.  This is true, for example, of
circulant graphs,  as illustrated in Section \ref{S:circu}.   In \cite{BKR}, it is shown that the adjacency matrix $\A$ for a McKay quiver
$\cR_{\VV}(\GG)$ can always be used to construct an {\it avalanche-finite matrix}, which has many interesting properties related to chip-firing dynamics.

 We fix a set $\{\cm_\mu\}_{\mu \in \Lambda(\GG)}$ of
conjugacy class representatives of $\GG$, and let $\mathcal C_{\mu}$ denote the conjugacy class of $\cm_\mu$.  Then $\cm_\zr$ is the
identity element, and $|\mathcal C_\zr | = 1$.     In this paper, we prove the following result giving a formula for the number of walks:

\begin{thm}{\rm (Theorem \ref{T:walk})} \, Assume $\VV$ is a finite-dimensional module over $\CC$  for the finite group $\GG$.   The number of walks of $k$-steps from node $\nu$ to node  $\lam$ on the McKay quiver  $\cR_{\VV}(\GG)$  is
\begin{align} (\A^k)_{\nu, \lam}  &=\vert \GG \vert^{-1}  \sum_{\mu \in\Lambda(\GG)}\vert \mathcal C_\mu \vert  \ \chi_{\nu}(\cm_\mu) \ \chi_{{}_{\VV}}(\cm_\mu)^k \ \overbar{\chi_{\lam}(\cm_\mu)} = \vert \GG \vert^{-1} \sum_{g \in \GG}
 \chi_{\nu}(g) \ \chi_{{}_{\VV}}(g)^k \ \overbar{\chi_{\lam}(g)}. \end{align}
\end{thm}
  As a consequence of this theorem,  we determine that the  Poincar\'e series for the number of walks from $\zr$ on $\lam$ on $\cR_\VV(\GG)$
(hence also for the multiplicities of the $\GG$-module  $\GG_\lam$ in the tensor powers  $\VV^{\ot k}$ and  for the dimensions of the centralizer algebra modules  $\dimm \Zs_k^\lam(\GG)$)  is given by

\begin{equation}\label{eq:poin1} \mathsf{P}^\lam(t)  = \sum_{k = 0}^\infty  (\A^k)_{\zr, \lam} \, t^k
=   \vert \GG \vert^{-1} \sum_{\mu \in\Lambda(\GG)}\vert \mathcal C_\mu \vert  \frac{\overbar{\chi_{\lam}(\cm_\mu)}}{ 1 - \chi_{{}_\VV}(\cm_\mu)t}  =   \vert \GG \vert^{-1}  \sum_{g \in\GG}   \frac{\overbar{\chi_{\lam}(g)}}{ 1 - \chi_{{}_\VV}(g)t}.
\end{equation}
Since the space $\mathsf{T}(\VV)^\GG = \bigoplus_{k=0}^\infty (\VV^{\ot k})^\GG$  of $\GG$-invariants in $\mathsf{T}(\VV) = \bigoplus_{k = 0}^\infty \VV^{\ot k}$ is the sum of the trivial $\GG$-summands
$\GG_\zr$ in $\mathsf{T}(\VV)$, it follows that  the Poincar\'e series for the tensor invariants  is  given by
\begin{equation}\label{eq:poin2} \mathsf{P}^\zr(t)
=   \vert \GG \vert^{-1}  \sum_{\mu \in\Lambda(\GG)}\vert \mathcal C_\mu \vert \  \frac{1}{ 1 - \chi_{{}_\VV}(\cm_\mu)t}
=   \vert \GG \vert^{-1}  \sum_{g\in \GG}  \frac{1}{ 1 - \chi_{{}_\VV}(g)t}.
\end{equation}
(An alternate derivation of \eqref{eq:poin2} can be found in \cite{DF}.)
 The results in \eqref{eq:poin1} and \eqref{eq:poin2} are tensor analogues of Molien's 1897 formulas for polynomials  that have
 played a prominent role in combinatorics, coding theory,  commutative algebra, and physics (see, for example, Stanley \cite{S1}, Sloane \cite{Sl}, Murai \cite{Mu}, and Forger \cite{Fo}).  To see this comparison, let
  $\{z_1,z_2, \ldots, z_n\}$ be a basis for $\VV$, and let  $\mathsf{S}(\VV) = \CC[z_1,z_2,\dots, z_n]$ be the symmetric algebra of polynomials
in the $z_i$.  Assume $\mathsf{S}_k(\VV)$ is the space of polynomials in $\mathsf{S}(\VV)$ of total degree $k$, and let
$\mathsf{S}_k^\lam(\VV)$ be the sum of all the copies of $\GG_\lam$ in $\mathsf{S}_k(\VV)$ (the $\lambda$-isotypic component).    According to \cite{Mo}, the Poincar\'e series are then given by the following expressions, which are similar to the ones for tensors:
\begin{align}  \mathsf{P}_{\mathsf{S}}^{\lam}(t)  & = \sum_{k = 0}^\infty \dimm \mathsf{S}_k^\lam(\VV) \, t^k
=   \vert \GG \vert^{-1}  \sum_{\mu \in\Lambda(\GG)} \vert \mathcal C_\mu \vert   \frac{\overbar{\chi_{\lam}(\cm_\mu)}}
{\mathsf{det}_\VV(\mathrm{I} - t\cm_\mu)} =   \vert \GG \vert^{-1}  \sum_{g \in \GG} \frac{\overbar{\chi_{\lam}(g)}}
{\mathsf{det}_\VV(\mathrm{I} - t g)},  \\
\mathsf{P}_{\mathsf{S}}^{\zr}(t)  &=   \vert \GG \vert^{-1}  \sum_{\mu \in\Lambda(\GG)} \vert \mathcal C_\mu \vert \frac{1} {\mathsf{det}_\VV(\mathrm{I} - t \cm_\mu)} =  \vert \GG \vert^{-1}  \sum_{g \in\GG} \frac{1}{ \mathsf{det}_\VV(\mathrm{I} -t g)}.  \end{align}

From \eqref{eq:gtens},  we note that
\begin{equation}\label{eq:eigen} \sum_{\lam \in \Lambda(\GG)} a_{\nu,\lam}\,\, \chi_\lam(\cm_\mu) = \chi_{{}_\VV}(\cm_\mu)\,\chi_\nu(\cm_\mu),
\end{equation}
 which implies that
the eigenvalues of the adjacency matrix $\A$ of $\cR_\VV(\GG)$ are
the character values $\chi_{{}_\VV}(\mathsf{c}_\mu)$ as  $\mu$ ranges over the elements of  $\Lambda(\GG)$, and the eigenvector corresponding to $\chi_{{}_\VV}(\mathsf{c}_\mu)$ is the column vector with entries $\chi_{\nu}(\mathsf{c_\mu})$ for $\nu \in  \Lambda(\GG)$.  The matrix of these eigenvectors is exactly the character table of $\GG$.  (Compare \cite[Sec.~1]{St} which considers the matrix $\mathsf{d} \mathrm{I}-\A$, where $\mathsf{d} =
\chi_{{}_\VV}(\cm_\zr) = \dimm \VV$.)

Theorem 2.1 of \cite{B2} shows that the Poincar\'e series $\mathsf{P}^\lam(t)$ can be expressed as a quotient of
two determinants under the assumption that
the module $\VV$ is isomorphic to its dual $\GG$-module.  But that assumption is unnecessary if the matrix $\A$ is replaced
by its transpose in computing the determinant in the numerator,  as in the statement below.  A proof of this result can be
deduced from the proposition in Appendix I, which holds for walks on arbitrary finite directed graphs.   In considering the rows
and columns of the adjacency matrix $\A$ in the next theorem,
we assume that the elements of $\Lambda(\GG)$ have been ordered in some fashion and that $\zr$ is always the first element
relative to that ordering.

\begin{thm}\label{T:main}  Let $\GG$ be a finite group with irreducible modules $\GG_\lambda$, $\lambda \in \Lambda(\GG)$,  over $\CC$, and let $\VV$ be a finite-dimensional $\GG$-module. Let $\A = \big(a_{\nu,\lambda}\big)$ be the adjacency matrix of the McKay quiver
$\cR_{\VV}(\GG)$,  and let  $\col^\lam$ be the
matrix $\mathrm{I}-t\A^{\tt T}$ with the column indexed by $\lam$ replaced by
$\delta_\zr = \left (\begin{smallmatrix}  1\\ 0 \\ \vdots \\  \\ 0 \end{smallmatrix}\right )$.
Then
\begin{equation}\label{eq:main} \mathsf{P}^\lam(t) =   \frac{\det(\col^\lam)}{\det(\mathrm{I}-t\A)} =
 \frac{\det(\col^\lam)}{\prod_{\mu \in \Lambda(\GG)} \left(1- \chi_{{}_\VV}(\cm_\mu)t\right)}. \end{equation}
\end{thm}

In \cite{Mc},  John McKay described a remarkable correspondence between the finite subgroups $\GG$   of the special unitary group  $\SU_2$  and the  simply laced affine Dynkin diagrams.  Almost a century earlier,  Felix Klein had determined that a finite subgroup of $\SU_2$ must be isomorphic to one of the following:  (a) a cyclic group $\ZZ_n =\ZZ/n \ZZ$ of order $n$, (b) a binary dihedral group $\DD_n$ of order $4n$, or (c) one of the 3 exceptional groups: the binary tetrahedral group $\TT$ of order 24,  the binary octahedral group $\OO$ of order 48, or the binary icosahedral group $\II$ of order 120.   McKay's observation was that the graph  $\cR_\VV(\GG)$  for $\GG =
\ZZ_n, \DD_n, \TT, \OO, \II$
relative to  its  defining representation $\VV= \mathbb C^2$
corresponds exactly to the affine Dynkin diagram $\mathsf{\hat{A}}_{n-1}$, $\mathsf{\hat{D}}_{n+2}$, $\mathsf{\hat{E}}_6$, $\mathsf{\hat{E}}_7$, $\mathsf{\hat{E}}_8$, respectively, where the node labeled by $\zr$ corresponding to the trivial $\GG$-module is the affine node.   The matrix  $\mathsf{C} = 2\, \mathrm{I} - \A$, where
$\A$ is adjacency matrix of $\cR_\VV(\GG)$,  is the associated affine Cartan matrix.
In this case, the Poincar\'e series for the tensor invariants in Theorem \ref{T:main}  specializes to the following:

\begin{thm} \label{T:mainsu}\,{\rm \cite[Thm. 3.1]{B2}} \ Let $\GG$ be a finite subgroup of $\SU_2$ and $\VV= \CC^2$.  Then the Poincar\'e series for the $\GG$-invariants $\mathsf{T}(\VV)^{\GG}$  in $\mathsf{T}(\VV) = \bigoplus_{k= 0}^\infty \VV^{\ot k}$ is

\begin{equation}\label{eq:inv}
\mathsf{P}^{\zr}(t) =  \frac{\det \left(\mathrm{I} - t \angstrom \right)}{\det \left(\mathrm{I} - t \A \right)}
=  \frac{\det \left(\mathrm{I} - t  \angstrom \right)}{\prod_{\mu \in \Lambda(\GG)} \left(1- \chi_{{}_\VV}(\cm_\mu)t \right)},
\end{equation}
where $\A$ is the adjacency matrix of the graph $\cR_{\VV}(\GG)$ (i.e. the affine Dynkin diagram corresponding to $\GG$), and  \ $\angstrom$ is the adjacency matrix of the finite Dynkin diagram
obtained by removing the affine node.
\end{thm}

As shown in \cite[Sec.~3]{B2}, the eigenvalues of $\angstrom$ and $\A$ are related to the exponents of the finite and affine root systems respectively, and the determinants in this formula can be expressed as Chebyshev polynomials of the second kind.    Results in a similar vein for the doubly laced root systems  can be found in \cite{B1}.

 We illustrate the usefulness of the results in our paper  by computing several examples,  as described below for various choices of $\GG$ and $\VV$.
 We obtain new expressions for the dimensions of the tensor invariants, the multiplicities of irreducible summands,  and the dimensions of centralizer algebras
 and their irreducible modules and the related exponential generating functions and Poincar\'e series using the methods presented here.

 When $\GG$ is abelian, the conjugacy classes consist of a single element of $\GG$, so we will always identify
 $\Lambda(\GG)$ with $\GG$  when $\GG$ is abelian.  Here is a brief summary of the examples studied in this work and the results shown for them.

 \begin{enumerate}
 \item  \emph{$\GG = \ZZ_r$ (a cyclic group of order $r$) and $\VV = \GG_{1} \oplus \GG_{r-1}$}: \newline
  In Section 3.1, we obtain a formula for the number of walks of $k$ steps on a circular graph with $r$ nodes.

 \item  \emph{$\GG = \ZZ_{13}$  and $\VV  = \bigoplus _{j}
  \GG_{j}$, where  $j = 1,3,4,9,10,12$}: \newline
  As shown in Section \ref{S:circu}, this example leads to an expression for the number of walks on the Paley graph ${\mathcal P}_{13}$ of order 13.   Paley graphs arise in studying quadratic residues in finite fields, and the key fact germane to the results here is that
  Paley graphs are circulant graphs (their adjacency matrices are circulant matrices).    The same method
  used for ${\mathcal P}_{13}$  can be applied to compute walks on $\text{\underline{any}}$ circulant graph.

  In Section \ref{S:Gauss}, we adopt a different approach and  determine closed-form formulas  for the number of walks of $k$ steps from $\zr$ to any node on a Paley (di)graph $\mathcal P_{p}$  of order  $p$ for an arbitrary odd prime $p$ using Theorem \ref{T:walk} and number-theoretic properties of Gauss sums.  When
  $p \equiv 1 \modd 4$,  ${\mathcal P}_p$ is an undirected graph, and when $p \equiv 3 \modd 4$, ${\mathcal P}_p$ is a directed graph (digraph).

  \item  \emph{$\GG = \Sr_n$, the symmetric group on $n$ letters, and $\VV$ is its $n$-dimensional permutation
  module}: \newline  Our results here lead to a proof of the relation
\begin{equation}\label{eq:sn}\dimm \Zs_k(\Sr_n) =  (n\,!)^{-1} \sum_{\sigma \in \Sr_n} \mathsf{F}(\sigma)^{2k}  =  \sum_{\ell=0}^n
\genfrac{\{}{\}}{0pt}{}{2k}{\ell}
\end{equation}
  between the number of fixed points $\mathsf{F}(\sigma)$ of permutations $\sigma$, and
  the Stirling numbers
  $\genfrac{\{}{\}}{0pt}{}{2k}{\ell}$
of the second kind, which count the number of ways
  to partition a set of $2k$ objects into $\ell$ nonempty disjoint parts.  (Note that
  $\genfrac{\{}{\}}{0pt}{}{0}{\ell}= 0$ unless $\ell =0$,
in which case it is 1.)  The first relation in \eqref{eq:sn} was shown by Farina and  Halverson in \cite{FaH} under the additional assumption that $n \geq 2k$ using the orthogonality of the characters of the partition algebra $\mathsf{P}_k(n)$, which
  is isomorphic to  the centralizer algebra $\Zs_k(\Sr_n) = \End_{\Sr_n}(\VV^{\ot k})$ when $n \geq 2k$.   By instead using the character theory
 of the symmetric group $\Sr_n$, we are able to deduce the first equality in \eqref{eq:sn} readily from our results
 without imposing restrictions on $n$ and $k$.

 The partitions $\lam$ of $n$ index the irreducible $\Sr_n$-modules.  Applying  Corollary \ref{C:centdim}\,(i) and \cite[Thm.~5.5(a)]{BHH}, we determine that
 \begin{equation}\label{eq:snlam} \dimm \Zs_k^\lam(\Sr_n) =  (n!)^{-1} \sum_{\sigma \in \Sr_n}
 \mathsf{F}(\sigma)^{k}\,\overbar{\chi_{\lam}(\sigma)}  =
 \sum_{\ell=0}^{n}
 \genfrac{\{}{\}}{0pt}{}{k}{\ell}  \, \mathsf{K}_{\lam, (n-\ell, 1^\ell)},
 \end{equation}
 where $ \mathsf{K}_{\lam, (n-\ell, 1^\ell)}$ \,  is the  \emph{ Kostka number}, and $(n-\ell, 1^\ell)$ is the
 partition of $n$ with one part of size $n-\ell$ and $\ell$ parts of size 1.   Equation \eqref{eq:sn} is a special
 case of \eqref{eq:snlam}, since $\dimm \Zs_k(\Sr_n)= \dimm \Zs_{2k}^\zr(\Sr_n)$,
 and the relevant Kostka numbers are all 1 in this case.   It follows from \eqref{eq:snlam} with $\lam = \zr$  that the dimension of the $\Sr_n$-invariants in $\VV^{\ot k}$ is given by
 \begin{equation}\label{eq:sninv1}
 \dimm (\VV^{\ot k})^{\Sr_n} =   (n!)^{-1} \sum_{\sigma \in \Sr_n}
 \mathsf{F}(\sigma)^{k} = \sum_{\ell=0}^n
 \genfrac{\{}{\}}{0pt}{}{k}{\ell} ,
 \end{equation}
 and the Poincar\'e series for the tensor invariants is given by
 \begin{equation}\label{eq:sninv2}
 \mathsf{P}^\zr(t) = \sum_{k=0}^\infty   \dimm (\VV^{\ot k})^{\Sr_n}\, t^k =   (n!)^{-1} \sum_{\sigma \in \Sr_n} \frac{1}{1- \mathsf{F}(\sigma) t}.
 \end{equation}
  It would be nice to have a bijective combinatorial proof
 of the identity in \eqref{eq:snlam}.

   \item  \emph{$\GG = \ZZ_r \wr \Sr_n$ (the wreath product) and $\VV$ is its $n$-dimensional module over $\CC$
  on which $\GG$ acts by $n \times n$ monomial matrices with entries of the form $\om^j$ for $j = 0,1,\dots, r-1$,
  where $\om$ is a primitive $r$th root of unity for $r \geq 2$}: \newline
  In Theorem \ref{T:wreath}\,(a) below,  we prove that
   \begin{equation*}
   \dimm (\VV^{\ot k})^\GG = \frac{1}{r^nn\,!} \sum_{m=1}^n  r^m \Fs_n(m)^k\left( \sum_{\ell_1,\ell_2,\ldots, \ell_m}  \binom{k}{\ell_1,\ell_2, \dots, \ell_m}
   \right),
   \end{equation*}
where the inner sum of multinomial coefficients is over all $0 \leq \ell_1,\ell_2, \dots, \ell_m\leq k$ such that $\ell_1 +\ell_2+ \cdots + \ell_m = k$ and $\ell_1 \equiv \ell_2 \equiv  \cdots
  \equiv \ell_m \equiv 0  \, \modd r$, and
$ \Fs_n(m) = \displaystyle{\frac{n!}{m!} \sum_{j=0}^{n-m} \frac{(-1)^j}{ j\,!}}$ is the number of permutations in $\Sr_n$ with
exactly $m$ fixed points.  Part (b) of Theorem \ref{T:wreath} gives an expression for the exponential generating function
$ \gr^\zr(t)= \displaystyle{\sum_{k=0}^\infty    \dimm (\VV^{\ot k})^\GG \, \frac{t^k}{k!}}$ in terms of a generalized hyperbolic function.
 Equation \eqref{eq:inv1} gives a second expression for the dimension of the  invariants
using the fact that the irreducible modules for $\GG = \ZZ_r \wr \Sr_n$ are indexed by $r$-tuples $\bal = (\alpha^{(1)},\alpha^{(2)}, \dots, \alpha^{(r)})$ of partitions $\alpha^{(i)}$ with $\sum_{i=1}^r  | \alpha^{(i)}| = n$:
\begin{equation}\qquad \dimm (\VV^{\ot k})^\GG
= \sum_{\bal \in \Lambda(\GG)} \, \,  \frac{\Big(\sum_{i=1}^r  \mathsf{F}(\alpha^{(i)})\, \om^{i-1}\Big)^k}
{r^{\mathsf{p}(\bal)} \,\prod_{j =1}^n    j^{\mathsf{p}_{j}(\bal) } \, \left(\prod_{i=1}^r  \mathsf{p}_{j}(\alpha^{(i)})!\right)} \quad \hbox{\rm for $\GG = \ZZ_r \wr \Sr_n$}. \end{equation}
In this formula  $\mathsf{p}_j(\alpha^{(i)})$ is the number of parts of $\alpha^{(i)}$ of
size $j$;   $\mathsf{p}_j(\bal) = \sum_{i=1}^r \mathsf{p}_j(\alpha^{(i)})$;  $\mathsf{p}(\bal) = \sum_{j=1}^n  \mathsf{p}_j(\bal)$;   and
$\mathsf{F}(\alpha^{(i)}) = \mathsf{p}_1(\alpha^{(i)})$, the number of parts of $\alpha^{(i)}$ of size 1, (the number of fixed points of a permutation
with cycle type  $\alpha^{(i)}$).       It is desirable to have a direct proof of the equivalence of these two formulas for $\dimm (\VV^{\ot k})^\GG$.

 When  $r = 2$,  the group $\GG = \ZZ_2 \wr \Sr_n$ is the Weyl group corresponding to the
  root systems $\mathrm{B}_n$ and $\mathrm{C}_n$.     In this case, the exponential generating function
  $\gr^\zr(t)$ for the $\GG$-invariants  in Theorem \ref{T:wreath}  is a linear combination of powers of hyperbolic cosines.
This result can be found in part (b) of Theorem \ref{T:WgroupBC}.   Part (a) of that result presents formulas for the dimensions of the tensor invariants from the perspective of
equation \eqref{eq:inv1}.

\item \emph{$\GG$ is the general linear group $\GL_2(\FF_q)$  of invertible $2 \times 2$ matrices over a finite field $\FF_q$ of $q$ elements, where $q$ is odd,
or $\GG$ is the special linear subgroup $\SL_2(\FF_q)$ of  matrices of determinant 1.
 The $\GG$-module $\VV$ is the $(q+1)$-dimensional
module over $\CC$ obtained by inducing the trivial module for the Borel subgroup $\mathsf{B}$ of upper-triangular matrices
in $\GG$}:  \newline
The module $\VV$ decomposes as a $\GG$-module,  $\VV = \GG_0 \oplus \VV_q$,   where $\GG_0$ is the trivial $\GG$-module and
$\VV_q$ is the $q$-dimensional irreducible Steinberg module.     In Theorems \ref{T:GL2q} and \ref{T:SL2q},  we derive formulas
for the dimension of the spaces $(\VV^{\ot k})^\GG$ and $(\VV_q ^{\ot k})^\GG$ of $\GG$-invariants  and determine the
Poincar\'e series for the tensor invariants $\mathsf{T}(\VV)^\GG$ and $\mathsf{T}(\VV_q)^\GG$.

   \item  \emph{ $\GG$ is an arbitrary finite abelian group,  say $\GG = \ZZ_{r_1} \times \cdots \times  \ZZ_{r_n}$,
 and $\VV = \GG_{\ve_1} \oplus \GG_{\ve_2} \oplus \cdots \oplus \GG_{\ve_n}$,   where $\varepsilon_j$ is the element of
 $\GG$ with 1 as its $j$th component  and $0$ as its other components}:  \newline
 In Section \ref{S:abelian}, we show that the \emph{exponential}
 generating function for the number of walks on the McKay quiver (equivalently, for the multiplicities
 of the irreducible $\GG$-modules in $\VV^{\ot k}$;  also, for the dimensions of the
 irreducible modules $\Zs_k^\lam(\GG)$ for the centralizer algebra $\Zs_k(\GG)$), is a product of generalized hyperbolic
 functions.   We deduce that the number of walks can be expressed as a sum of multinomial
 coefficients.   When $r_1 = r_2 = \cdots = r_n = 2$,  we obtain a formula for the number of walks on
 a hypercube of dimension $n$ and the expression for the exponential generating function for the number
 of walks as a product of hyperbolic sines and cosines that was given in \cite[Cor.~4.29]{BM}.   In Sections \ref{S:abelbasis} and \ref{S:abeldiag},
 we exhibit a basis for $\Zs_k(\GG)$ and view $\Zs_k(\GG)$ as a diagram algebra by  giving a diagrammatic realization of
 the basis elements.
 \end{enumerate}

\section {Walks and Poincar\'e series}

\subsection{Expressions for counting walks, multiplicities, and centralizer algebra dimensions}
There is a Hermitian  inner product on the class functions of a finite group $\GG$
defined by
$$\langle \phi,\psi \rangle = \vert \GG \vert^{-1}  \sum_{g \in \GG} \phi(g) \overbar{\psi(g)}
=  \vert \GG \vert^{-1}  \sum_{\mu \in \Lambda(\GG)}\vert \mathcal C_\mu \vert \, \phi(\cm_\mu) \overbar{\psi(\cm_\mu)},$$
where ``$-$" denotes the complex conjugate.
The irreducible  characters   $\chi_\lam$ for $\lam \in \Lambda(\GG)$
 satisfy the well-known orthogonality relations relative to this inner product
(see for example, \cite[(2.10) and Ex. 2.21]{FuH}):
 \begin{align} & \langle \chi_\nu, \chi_\lam \rangle = \vert \GG \vert^{-1}  \sum_{g \in \GG} \chi_\nu(g) \overbar{\chi_\lam(g)}= \delta_{\nu,\lam},  \label{eq:ortho1} \\
&\vert \GG \vert^{-1} \sum_{\lam \in \Lambda(\GG)} \chi_{\lam}(\cm_\mu)\chi_\lam(\cm_\nu) = \begin{cases} \vert \mathcal C_\mu \vert &\quad \text{ if \  $\mu = \nu$,}\\
0 &\quad \text{ if  \ $\mu \neq \nu$.}\end{cases} \label{eq:ortho2}  \end{align}
Therefore,  if $\mathsf{U}$ is a $\GG$-module over $\CC$ with character $\chi_{{}_\UU}$,   then \eqref{eq:ortho1} implies that
$$\langle \chi_{{}_\UU}, \chi_\lam\rangle  =  \vert \GG \vert^{-1}  \sum_{g \in \GG} \chi_{{}_\UU}(g) \overbar{\chi_\lam(g)}
=  \vert \GG \vert^{-1}  \sum_{\mu \in \Lambda(\GG)}\vert \mathcal C_\mu \vert \, \chi_{{}_\UU}(\cm_\mu) \overbar{\chi_\lam(\cm_\mu)} $$
 is the multiplicity of $\GG_\lam$ as a summand of $\UU$.     Applying this to the $\GG$-module $\UU = \GG_\nu \ot  \VV^{\ot k}$, which
 has character $\chi_{\nu}\chi_{{}_\VV}^k$,   gives the following result.
\medskip

\begin{thm}\label{T:walk} Assume $\VV$ is  finite-dimensional module for the finite group $\GG$.   The number of walks of $k$-steps from node $\nu$ to node  $\lam$ on the McKay quiver  $\cR_{\VV}(\GG)$ (equivalently,  the multiplicity of $\GG_\lam$ in $\GG_\nu \ot \VV^{\ot k}$) is equal to
\begin{align}\label{eq:walkcount1} \begin{split} (\A^k)_{\nu, \lam}  &=\vert \GG \vert^{-1}  \sum_{\mu \in\Lambda(\GG)}\vert \mathcal C_\mu \vert  \ \chi_{\nu}(\cm_\mu) \ \chi_{{}_{\VV}}(\cm_\mu)^k \  \overbar{\chi_{\lam}(\cm_\mu)}. \end{split} \end{align}
\end{thm}

\begin{cor}\label{C:centdim}  Under the hypotheses of Theorem \ref{T:walk}, the following hold:
\begin{itemize} \item[{\rm (i)}]   the dimension of the irreducible
module $\Zs_k^\lam(\GG)$ for the centralizer algebra   $\Zs_k(\GG) = \End_\GG(\VV^{\ot k})$
is given by
\begin{equation}\label{eq:centmoddim}
 \dimm  \Zs_k^\lam(\GG) =  (\A^k)_{\zr,\lam} =  \vert \GG \vert^{-1}  \sum_{\mu \in\Lambda(\GG)}\vert \mathcal{C} _\mu \vert \, \chi_{{}_{\VV}}(\cm_\mu)^k \  \overbar{\chi_{\lam}(\cm_\mu)}  =\vert \GG \vert^{-1}  \sum_{g \in \GG}  \ \chi_{{}_{\VV}}(g)^k \
 \overbar{\chi_{\lam}(g)};  \end{equation}
 \item[{\rm (ii)}]  the dimension of the space of $\GG$-invariants in $\VV^{\ot k}$ is
 \begin{equation} \dimm (\VV^{\ot k})^\GG = (\A^{k})_{\zr,\zr} = \vert \GG \vert^{-1}  \sum_{\mu \in\Lambda(\GG)}\vert  \mathcal{C}_\mu \vert \ \chi_{{}_{\VV}}(\cm_\mu)^{k}  =\vert \GG \vert^{-1}  \sum_{g \in \GG} \ \chi_{{}_{\VV}}(g)^{k}; \quad \text{and} \end{equation}
 \item[{\rm (iii)}]  when $\VV$ is a self-dual $\GG$-module, \end{itemize}
\begin{equation} \label{eq:centdim}\qquad \dimm  \Zs_k(\GG) = \dimm \Zs_{2k}^\zr(\GG) = (\A^{2k})_{\zr,\zr} = \vert \GG \vert^{-1}  \sum_{\mu \in\Lambda(\GG)}\vert  \mathcal{C}_\mu \vert \ \chi_{{}_{\VV}}(\cm_\mu)^{2k}  =\vert \GG \vert^{-1}  \sum_{g \in \GG} \ \chi_{{}_{\VV}}(g)^{2k}.\end{equation}
  \end{cor}

\begin{subsection}{Poincar\'e series} \end{subsection}
It is a consequence of the results in \eqref{eq:centmoddim} and \eqref{eq:centdim} that the Poincar\'e series
\begin{equation} \Ps^\lam(t) := \sum_{k = 0}^\infty (\A^k)_{\zr,\lam}\ t^k  =  \sum_{k = 0}^\infty \ms_k^\lam\ t^k =
\sum_{k = 0}^\infty \dimm \Zs_k^\lam(\GG)\ t^k \end{equation}   has the following expression

\begin{align} \Ps^\lam(t) &=   \vert \GG \vert^{-1}  \sum_{\mu \in\Lambda(\GG)}\vert \mathcal{C}_\mu \vert  \frac{\overbar{\chi_{\lam}(\cm_\mu)}}
{1 - \chi_{{}_\VV}(\mathsf{c}_\mu)t}
=   \vert \GG \vert^{-1}   \sum_{g \in \GG} \frac{\overbar{\chi_{\lam}(g)}}
{1 - \chi_{{}_\VV}(g)t} \label{eq:Molien1} \\
&=   \frac{\det(\col^\lam)}{\det(\mathrm{I}-t\A)} =
 \frac{\det(\col^\lam)}{\prod_{\mu \in \Lambda(\GG)} \left(1- \chi_{{}_\VV}(\cm_\mu)t\right)}, \label{eq:Molien2}  \end{align}
where $\col^\lam$ is the matrix $\mathrm{I} - t \A^{\tt T}$ with the column
indexed by $\lam$ replaced by
$\delta_\zr = \left (\begin{smallmatrix}  1\\ 0 \\ \vdots \\  \\ 0 \end{smallmatrix}\right )$ as in Theorem \ref{T:main}.
Then a special case of this formula is the Poincar\'e series for the tensor invariants $\mathsf{T}(\VV)^\GG$  in $\mathsf{T}(\VV) = \bigoplus_{k = 0}^\infty \VV^{\ot k}$:
\begin{align} \Ps^\zr(t) &=   \vert \GG \vert^{-1}  \sum_{\mu \in\Lambda(\GG)}\vert \mathcal{C}_\mu \vert  \frac{1}
{1 -\chi_{{}_\VV}(\mathsf{c}_\mu)t}
=   \vert \GG \vert^{-1}   \sum_{g \in \GG} \frac{1}
{1 - \chi_{{}_\VV}(g)t} \label{eq:point} \\
&=   \frac{\det(\col^\zr)}{\det(\mathrm{I}-t\A)} =
 \frac{\det(\col^\zr)}{\prod_{\mu \in \Lambda(\GG)} \left(1- \chi_{{}_\VV}(\cm_\mu)t\right)}. \label{eq:point2} \end{align}
 The expressions in \eqref{eq:Molien1} and \eqref{eq:point} are analogs of Molien's formulas
 \begin{align}  \mathsf{P}_{\mathsf{S}}^{\lam}(t)  & = \sum_{k = 0}^\infty \dimm \mathsf{S}_k^\lam(\VV) \, t^k
=   \vert \GG \vert^{-1}  \sum_{\mu \in\Lambda(\GG)} \vert \mathcal C_\mu \vert   \frac{\overbar{\chi_{\lam}(\cm_\mu)}}
{\mathsf{det}_\VV(\mathrm{I} - t \cm_\mu)} =  \vert \GG \vert^{-1}  \sum_{g \in \GG}  \frac{\overbar{\chi_{\lam}(g)}}
{\mathsf{det}_\VV(\mathrm{I} - t g)},  \\
\mathsf{P}_{\mathsf{S}}^{\zr}(t)  &=   \vert \GG \vert^{-1}  \sum_{\mu \in\Lambda(\GG)} \vert \mathcal C_\mu \vert \frac{1} {\mathsf{det}_\VV(\mathrm{I} - t\cm_\mu)} =  \vert \GG \vert^{-1}  \sum_{g \in\GG} \frac{1}{ \mathsf{det}_\VV(\mathrm{I} - tg)}.  \end{align}
for multiplicities of $\GG$-modules and invariants in polynomials, as described in the Introduction.

%%%%%%%%%%%%%%%%%%%%%%%%%%%%%%%%%%%%%%

\section{Cyclic examples}
\subsection{$\GG =\ZZ_r$}\label{SS:cyclic}
When  $\GG = \ZZ_r=\ZZ/r\ZZ$, we identify the elements of  $\Lambda(\GG)$ with  the elements $\{0,1,\dots, r-1\}$  of $\ZZ_r$.
Then for $a \in \GG$,  the character $\chi_a$ of $\GG_{a}$ is given by
$\chi_a(b) = \om^{ab}$ for $a,b\in \GG$, where $\om = \er^{2\pi i/r}$.    We assume  $\VV  =\GG_{1}
\oplus \GG_{r-1}$.   The McKay quiver  $\cR_{\VV}(\ZZ_r)$ is a circular graph with $r$ nodes, and a step from a node  on the graph amounts to moving one step to the left or to the right.
Then for $b \in \GG$, we have $\chi_{{}_{\VV}}(b) = \chi_{1}(b) + \chi_{r-1}(b)  =
\om^{b} + \om^{-b} = 2 \mathsf{cos}(2\pi i b/r)$.    Therefore
$$\chi_{{}_{\VV^{\ot k}}}(b)  = \chi_{{}_{\VV}}(b)^k = (\om^{b} + \om^{-b})^k = \sum_{\ell=0}^k \binom{k}{\ell}  \om^{(k-\ell)b} \om^{-\ell b}=
\sum_{\ell=0}^k \binom{k}{\ell}  \om^{(k-2\ell)b}.$$

Now using the fact that
\begin{equation}\label{eq:omsum}  \sum_{b= 0}^{r-1}  \om^{mb} = \begin{cases} r  & \quad \hbox{\rm if  \ $m \equiv 0 \, \modd r$,} \\
0 & \quad \hbox{\rm otherwise,}\end{cases} \end{equation}
and Theorem \ref{T:walk},  we have the following expression for the number of walks of $k$ steps from $a$ to $c$ on $\cR_{\VV}(\ZZ_r)$:
\begin{align} \begin{split} \label{eq:cycwalk}
 (\A^k)_{a,c} &=r^{-1}\sum_{b \in \ZZ_r} \chi_{a}(b)
\chi_{{}_{\VV}}(b)^k \overbar{\chi_c(b})
= r^{-1} \sum_{b= 0}^{r-1} \om^{(a-c)b} \sum_{\ell=0}^k  \binom{k}{\ell} \om^{(k-2\ell)b} \\
&=r^{-1} \sum_{\ell=0}^k \binom{k}{\ell} \sum_{b=0}^{r-1} \om^{(k-2\ell+a-c)b}
 = \sum_{\substack{0 \leq \ell \leq k  \\ k-2\ell\, \equiv \, c-a \modd r}}  \binom{k}{\ell}. \end{split} \end{align}
 Therefore, the dimension of the irreducible module $\Zs_k^c(\ZZ_r)$ for  the centralizer algebra $\Zs_k(\ZZ_r) = \End_{\ZZ_r}\left(\VV^{\ot k}\right)$ is
 $$\dimm \Zs_k^c(\ZZ_r)  =  (\A^k)_{\zr,c} =\sum_{\substack{0 \leq \ell \leq k \\ k-2\ell\, \equiv \, c \modd r}}  \binom{k}{\ell}.$$
 In particular,  in order for the irreducible $\ZZ_r$-module labeled by $c$ to occur in $\VV^{\ot k}$ with multiplicity
 at least one, equivalently,  in order for $\dimm \Zs_k^c(\ZZ_r)$ to be nonzero,  it must be that
 $k-c \equiv  2\ell \modd r$ for some $\ell$.   Let $\ell_c$ be the least nonnegative integer with that property.   Then
$$\dimm \Zs_k^c(\ZZ_r)  = \sum_{\substack{0 \leq \ell \leq k  \\ \ell\, \equiv \, \ell_c \modd \tilde r}}
\binom{k}{\ell},$$
where $\tilde r = r$ if $r$ is odd,  and $\tilde r = r/2$ if $r$ is even. Since the module $\VV$ is self dual,
$$\dimm \Zs_k(\ZZ_r) = \dimm \Zs_{2k}^\zr(\ZZ_r) =  \sum_{\substack{0 \leq \ell \leq 2k \\ k-\ell \equiv 0 \modd \tilde r}}  \binom{2k}{\ell}.$$
(Compare  \cite[Thm.~2.17(i) and Thm.~2.8(d)]{BBH}.)  These
formulas can be interpreted as computing Pascal's triangle on a cylinder of
diameter $\tilde r$.  (See \cite[Sec.~4.2]{BBH} for more details.)

Here is a specific example to demonstrate the above results. \medskip

\begin{example}{\rm When $k = 6$ and $r = 10$,
\begin{align*} \dimm \Zs_6(\ZZ_{10}) & =  \sum_{\substack{0 \leq \ell \leq 12 \\ 6-\ell \, \equiv\, 0 \modd 5}}
\binom {12}{\ell}\\ &= \binom{12}{1} + \binom{12}{6} + \binom{12}{11} = 12 + 924 + 12 = 948. \end{align*}
This can be seen from the Bratteli diagram for the cyclic group of order 10  (which can be found in Appendix II
of this paper and in \cite[Sec.~4.2]{BBH}).   The right-hand column there
displays the dimension of the centralizer algebra.   Since the dimension of the irreducible module $\Zs_6^{8}(\ZZ_{10})$ is
the number of walks of 6 steps from 0 to 8 on the McKay quiver for $\GG = \ZZ_{10}$ and $\VV = \GG_1 \oplus \GG_9$, we have from \eqref{eq:cycwalk},
$$\dimm \Zs_6^{8}(\ZZ_{10}) =  \sum_{\substack{0 \leq \ell \leq 6 \\  6-2\ell \, \equiv\, 8 \modd 5}}
\binom{6}{\ell} =
\binom{6}{4} = 15.$$  This is the subscript on the node labeled  8 on level 6 of the Bratteli diagram for
the cyclic group of order 10.}
\end{example}

\begin{subsection}{Circulant graphs}\label{S:circu} \end{subsection}
The Paley graphs are a family of graphs constructed from quadratic residues in finite fields.
The Paley graph $\mathcal{P}_{13}$  of order 13 is pictured below.    Every Paley graph
is a circulant graph, which is equivalent to saying its adjacency matrix is a circulant matrix.    There
are many different characterizations of circulant graphs and circulant matrices.  (The article by Kra and Simanca \cite{KS} nicely summarizes many of them.)
Most relevant here is the fact that a graph is circulant if and only if its automorphism group
contains a cyclic group acting transitively on its nodes.   For $\mathcal{P}_{13}$ this group is
$\ZZ_{13}$.   In the notation of the previous example, we can take the module $\VV$ so
that  $\chi_{{}_\VV} = \sum_j  \chi_j$,  where the
sum is over $j =1,3,4,9,10,12$ (the quadratic residues mod 13).        Then a step on $\mathcal{P}_{13}$ corresponds to tensoring with
this particular choice of  $\ZZ_{13}$-module $\VV$.      Using that fact and  Theorem \ref{T:walk},   we have the following
(where $\om$ is a primitive $13$th root of 1):
\begin{figure}[ht!]
\[\includegraphics[scale=.85]{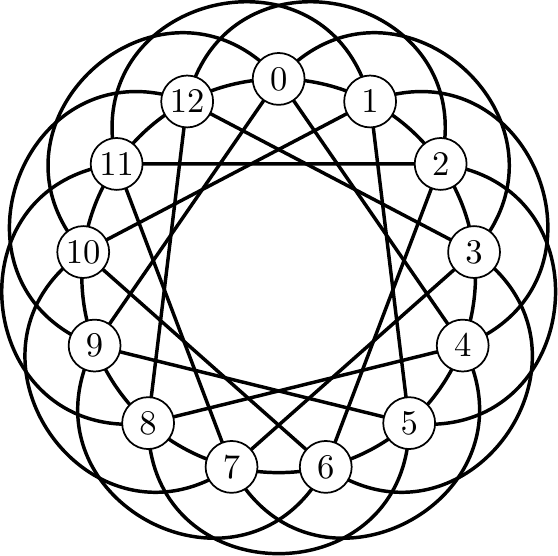}\]
\caption{Paley graph $\mathcal{P}_{13}$} \label{fig:P13}
\end{figure}
\medskip

\begin{cor}\label{C:Paley}   The number of walks of $k$ steps from $\zr$ to $c \in \{0,1,\dots, 12\}$  on the Paley graph $\mathcal{P}_{13}$ is
\begin{align*}\big(\mathsf{A}^k\big)_{\zr,c} & = \displaystyle{ (13)^{-1}
\sum_{\substack{0 \leq \ell_1,\ell_2, \dots, \ell_6\leq k \\ \ell_1 + \ell_2 + \cdots+\ell_6 = k}}
\binom{k}{\ell_1,\ell_2,\ldots,\ell_6}
\left(\sum_{b=0}^{12} \omega^{(\ell_1+3\ell_2+4\ell_3 + 9\ell_4 + 10\ell_5+12\ell_6-c)b}\right)} \\
& =\displaystyle{  \sum_{\substack{0 \leq \ell_1,\ell_2, \ldots \ell_6 \leq k, \ \, \ell_1+\ell_2+\cdots+\ell_6 = k \\ \ell_1+3\ell_2+\cdots+12\ell_6 \, \equiv \, c \, \mathsf{mod}\,13}} \binom{k}{\ell_1,\ell_2,\dots,\ell_6}}. \end{align*}
\end{cor}

Walks on any circulant graph can be enumerated by exactly the same type of argument.

 \subsection{Paley (di)graphs $\mathcal P_p$ of order $p$ an odd prime} \label{S:Gauss}
Suppose $p$ is an odd prime and $\om = \er^{2\pi i/p}$.     The nodes in the Paley (di)graph $\mathcal P_p$ are labeled
by the elements in $\{0,1,\dots, p-1\}$, and the ones connected to $\zr$ are labeled by the distinct square values $x^2$  in
 $\ZZ_p^\times = \{1,2, \dots,p-1\}$ (the quadratic residues modulo $p$).   As noted earlier, for $p = 13$ these are the values $x^2 = 1,3,4,9,10,12$.    When $p \equiv 1 \modd 4$, ${\mathcal P}_p$ is an undirected graph, and
for $p \equiv 3 \modd 4$  it is a digraph, as illustrated below for $p = 7$.

\begin{figure}[ht!]
\[\includegraphics[scale=1]{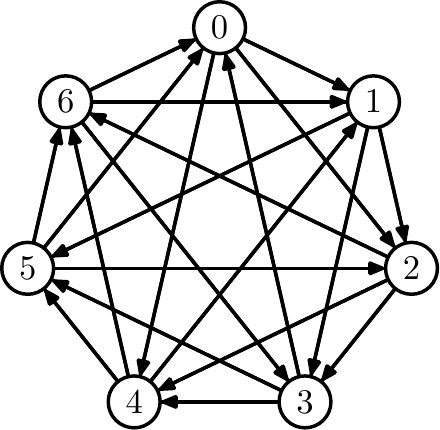}\]
\caption{Paley digraph ${\mathcal P}_7$} \label{fig:P7}
 \end{figure}

We take $\VV$ so that $\cR_\VV(\ZZ_p)$  is $\mathcal P_p$.   Then
$$\chi_{{}_\VV}(b) = f(b):= \sum_{x^2 \in \ZZ_p^\times}  \om^{bx^2}$$
for $b \in \ZZ_p$, and we know from  \eqref{eq:centmoddim} that the number of walks of $k$ steps from $\zr$ to $c$ on the graph ${\mathcal P}_p$  is given by
\begin{equation}\label{eq:Paleyp}  (\A^k)_{\zr,c} =  \frac{1}{p} \sum_{b \in \ZZ_p} \, \chi_{{}_{\VV}}(b)^k  \overbar{\chi_{c}(b)}= \frac{1}{p} \sum_{b=0}^{p-1} f(b)^k \om^{-cb}
\end{equation}
We evaluate this expression using well-known facts about Gauss sums, which can be found for example in
 \cite[Chap.~8]{IR}.    Suppose
\begin{equation}\label{eq:xi} \xi = \begin{cases}    1\quad& \qquad \hbox{\rm if \ $p \equiv 1\, \modd 4$,} \\
 i = \sqrt{-1}\quad& \qquad \hbox{\rm if \ $p \equiv 3 \, \modd 4$.}\end{cases}\end{equation}
The Gauss sum  $g(b)  = \sum_{x=0}^{p-1}  \om^{b x^2}$ equals $p$ when $b = 0$, and    for $b \in \ZZ_p^\times$
$$g(b) = \legendre{b}{p} g(1) = \begin{cases}
\ \,\xi\sqrt{p} &\quad \text{if $b$ is a quadratic residue modulo $p$,} \\
-\xi\sqrt{p} &\quad \text{if $b$ is a quadratic nonresidue modulo $p$,}
 \end{cases}$$
 where $\legendre{b}{p}$ is the Legendre symbol, which is $1$ if $b$ is a quadratic residue and $-1$ otherwise.
Since the number of quadratic residues equals the number of quadratic nonresidues,
it follows that
$$f(b)=\half \left( g(b) -1 \right)= \begin{cases}
\ \half(\xi\sqrt{p}-1) &\text{if $b$ is a nonzero quadratic residue modulo $p$,} \\
-\half (\xi\sqrt{p}+1) &\text{if $b$ is a quadratic nonresidue modulo $p$,}\\
\ \half(p-1) &\text{if $b = \zr$.}
 \end{cases}$$

 Our aim in this section is to prove

 \begin{thm}\label{T:paley}  Assume $\mathcal P_p$ is the Paley (di)graph of order $p$ a prime and $\xi$ is as in \eqref{eq:xi}.
 Then the number of walks of $k$ steps from $\zr$ to $c$  on $\mathcal P_p$ is given by one of the following:
 \begin{itemize}
 \item[{\rm (i)}]   If  $c$ is a nonzero quadratic residue,  then
 \begin{align*} \hspace{-.38cm} (\A^k)_{\zr,c}
 & = \begin{cases}  \frac{1}{2^{k+1}p} \left(2\left(p-1\right)^k  +\left(\sqrt{p} -1\right)^{k+1} +  (-1)^{k+1}\left(\sqrt{p}+1\right)^{k+1}\right) & \hspace{-.32cm} \hbox{\rm if $p \equiv 1\,\modd 4$,} \\
  \frac{1}{2^{k+1}p} \left( 2\left(p-1\right)^k  +(p+1)\left(i\sqrt{p} -1\right)^{k-1}   +  (-1)^k (p+1)\left(i\sqrt{p}+1\right)^{k-1} \right )&\hspace{-.32cm}  \hbox{\rm if $p \equiv 3 \modd 4$.}\end{cases}
  \end{align*}
    \item[{\rm (ii)}]   If  $c$ is a quadratic nonresidue,  then
 \begin{align*}  (\A^k)_{\zr,c}
 & = \begin{cases}  \frac{p-1}{2^{k+1}p} \left(2\left(p-1\right)^{k-1}  +\left(\sqrt{p} -1\right)^{k-1} +  (-1)^{k}\left(\sqrt{p}+1\right)^{k-1}\right) &  \hbox{\rm if \ $p \equiv 1\,\modd 4$} \\
 \frac{1}{2^{k+1}p} \left( 2\left(p-1\right)^k  -\left(i \sqrt{p} +1\right)^{k+1}   +  (-1)^{k+1}\left(i \sqrt{p}+1\right)^{k+1}\right )&
  \hbox{\rm if \  $p \equiv 3 \modd 4$.}\end{cases}
  \end{align*}
  \item[{\rm (iii)}]  If  $c = \zr$, then
  \begin{equation*} (\A^k)_{\zr,\zr} = \frac{p-1}{2^{k+1}p}
 \left( 2 \left(p-1\right)^{k-1}  + \left(\xi \sqrt{p} -1\right)^k   + (-1)^k\left( \xi \sqrt{p}+1\right)^k \right).
 \end{equation*}
 \end{itemize}
 \end{thm}

\begin{proof}    Since the quadratic nonresidues modulo $p$ are all of the form $ax^2$ for some fixed quadratic nonresidue $a$,
we have from \eqref{eq:Paleyp}
\begin{align}\begin{split}\label{eq:splitsum}  (\A^k)_{\zr,c}  & = \frac{1}{p} \left(  \left( \frac{p-1}{2}\right)^k  + \sum_{x^2 \in \ZZ_p^\times}\left(\frac{\xi \sqrt{p} -1}{2}\right)^k \om^{-x^2 c}  +  \sum_{x^2 \in \ZZ_p^\times}(-1)^k\left(\frac{\xi \sqrt{p}+1}{2}\right)^k \om^{-ax^2 c}\right)\\
  & =  \frac{1}{p} \left(  \left( \frac{p-1}{2}\right)^k  +\left(\frac{\xi \sqrt{p} -1}{2}\right)^k  \sum_{x^2 \in \ZZ_p^\times} \om^{-x^2 c}  +  (-1)^k\left(\frac{\xi \sqrt{p}+1}{2}\right)^k \sum_{x^2 \in \ZZ_p^\times} \om^{-ax^2 c}\right ).
\end{split}    \end{align}

Now if $c \neq \zr$,   then
$$g(-c) = \legendre{-c}{p} g(1)  =  \legendre{c}{p} \legendre{-1}{p}g(1)
= \begin{cases} \  g(c)  \quad& \quad \hbox{\rm if \ $p \equiv 1\, \modd 4$} \\
-g(c) \quad& \quad \hbox{\rm if \ $p \equiv 3 \, \modd 4$,}\end{cases}$$
so that
$$f(-c) = \begin{cases}  f(c)  \quad& \quad \hbox{\rm if \ $p \equiv 1\, \modd 4$,} \\
-(f(c)+1) \quad& \quad \hbox{\rm if \ $p \equiv 3\, \modd 4$.}\end{cases}$$
Therefore when $c \neq \zr$,
\begin{align}\begin{split}\label{eq:values}  (\A^k)_{\zr,c}
 & = \begin{cases}  \frac{1}{p} \left(  \left( \frac{p-1}{2}\right)^k  +\left(\frac{\sqrt{p} -1}{2}\right)^k f(c)  +  (-1)^k\left(\frac{\sqrt{p}+1}{2}\right)^k f(ac) \right ) & \  \hbox{\rm if \ $p \equiv 1\,\modd 4$} \\
 \frac{1}{p} \left(  \left( \frac{p-1}{2}\right)^k  -\left(\frac{i\sqrt{p} -1}{2}\right)^k (f(c)+1)  +  (-1)^{k+1}\left(\frac{i \sqrt{p}+1}{2}\right)^k(f(ac)+1)
 \right) & \ \hbox{\rm if \ $p \equiv 3\, \modd 4$.}\end{cases}
 \end{split} \end{align}
 We examine the expression in \eqref{eq:values}  for the scenarios in (i) and (ii) of Theorem \ref{T:paley}.

(i) When $c \in \ZZ_p^\times$ is a quadratic residue modulo $p$,  then
\begin{align*}  (\A^k)_{\zr,c}
 & = \begin{cases}  \frac{1}{2^{k+1}p} \left(2\left(p-1\right)^k  +\left(\sqrt{p} -1\right)^{k+1} +  (-1)^{k+1}\left(\sqrt{p}+1\right)^{k+1}\right) & \hspace{-.42cm} \hbox{\rm if $p \equiv 1\,\modd 4$,} \\
 \frac{1}{2^{k+1}p} \left( 2\left(p-1\right)^k  -\left(i\sqrt{p} -1\right)^k (i \sqrt{p}+1)   +  (-1)^{k+1}\left(i \sqrt{p}+1\right)^k\left(i \sqrt{p} -1\right) \right )  \\
=  \frac{1}{2^{k+1}p} \left( 2\left(p-1\right)^k  +(p+1)\left(i \sqrt{p} -1\right)^{k-1}   +  (-1)^{k}(p+1)\left(i \sqrt{p}+1\right)^{k-1}\right )&\hspace{-.42cm}  \hbox{\rm if $p \equiv 3 \modd 4$.}\end{cases}
  \end{align*}

(ii) When $c$ is a quadratic nonresidue modulo $p$,
\begin{align*}  (\A^k)_{\zr,c}
 & = \begin{cases}  \frac{p-1}{2^{k+1}p} \left(2\left(p-1\right)^{k-1}  +\left(\sqrt{p} -1\right)^{k-1} +  (-1)^{k}\left(\sqrt{p}+1\right)^{k-1}\right) &  \hbox{\rm if \ $p \equiv 1\,\modd 4$,} \\
 \frac{1}{2^{k+1}p} \left( 2\left(p-1\right)^k  -\left(i\sqrt{p} +1\right)^{k+1}   +  (-1)^{k+1}\left(i \sqrt{p}+1\right)^{k+1}\right )&
  \hbox{\rm if \  $p \equiv 3 \modd 4$.}\end{cases}
  \end{align*}

 (iii)  Finally, when $c = 0$,   then \eqref{eq:splitsum} implies
\begin{align*} (\A^k)_{\zr,\zr}  &  = \frac{1}{2^{k+1}p} \left(2\left(p-1\right)^k  +\left(\xi \sqrt{p} -1\right)^k (p-1)   +  (-1)^k \left( \xi \sqrt{p}+1\right)^k(p-1) \right)\\
&= \frac{p-1}{2^{k+1}p}
 \left( 2\left(p-1\right)^{k-1}  + \left(\xi \sqrt{p} -1\right)^k   + (-1)^k\left( \xi \sqrt{p}+1\right)^k \right),
 \end{align*}
 to give the assertion in part (iii).
\end{proof}
%%%%%%%%%%%%%%%%%%%%%%%%%%%%%%%%%%%%%%%%%%%%%%%%%

\section{The groups $\Sr_n$ and $\ZZ_r \wr \Sr_n$}
\begin{subsection}{The symmetric group $\Sr_n$}\end{subsection}
The irreducible modules for the symmetric group
$\Sr_n$ are in one-to-one correspondence with the partitions $\lambda \vdash n$, and the conjugacy classes
are determined by the cycle decomposition of the permutations, hence they also are indexed by the partitions of $n$.
If  $\VV$ is taken to be the $n$-dimensional permutation module on which $\Sr_n$ acts by permuting the basis
elements,   then for all $\sigma \in \Sr_n$,
\begin{equation}\chi_{{}_\VV}(\sigma)  = \mathsf{tr}_{\VV}(\sigma) =  \Fs(\sigma),\end{equation}
where $\Fs(\sigma)$ is the number of fixed points of $\sigma$.    As a result, we know from
\eqref{eq:point} and \eqref{eq:point2}  that the Poincar\'e series for the tensor invariants $\mathsf{T}(\VV)^{\Sr_n}$ is given by
\begin{align}\begin{split}\label{eq:Sninv} \Ps^\zr(t) &=  (n\,!)^{-1}  \sum_{\mu\vdash n}\vert \mathcal{C}_\mu \vert  \frac{1}
{1 - \Fs(\mathsf{c}_\mu)t}
=   (n\,!)^{-1}  \sum_{\sigma \in \Sr_n} \frac{1}
{1 - \Fs(\sigma)t}\\
&=   \frac{\det(\col^\zr)}{\det(\mathrm{I}-t\A)} =
 \frac{\det(\col^\zr)}{\prod_{\mu\,\vdash n} \left(1- \Fs(\cm_\mu)t\right)}
\end{split} \end{align}
where $\col^\zr$ and $\A$ are as in Theorem \ref{T:main}.
For the centralizer algebra
$\Zs_k(\Sr_n) = \End_{\Sr_n}(\VV^{\ot k})$ and its irreducible module $\Zs_k^\lam(\Sr_n)$,
\begin{align}\begin{split} \dimm \Zs_k^\lam(\Sr_n)  &= (n\,!)^{-1} \sum_{\sigma \in \Sr_n} \Fs(\sigma)^k\,\overbar{ \chi_\lam(\sigma)}, \\
\dimm \Zs_k(\Sr_n) & =  (n\,!)^{-1} \sum_{\mu \vdash n}\vert \mathcal{C}_\mu \vert \ \Fs(\mathsf{c}_\mu)^{2k} =
(n\,!)^{-1} \sum_{\sigma \in \Sr_n}  \Fs(\sigma)^{2k}. \end{split} \end{align}

The centralizer algebra $\Zs_k(\Sr_n)$  for the $\Sr_n$-action on the $k$-fold tensor power of its permutation module $\VV$
is a homomorphic image of the partition algebra $\mathsf{P}_k(n) \onto \Zs_k(\Sr_n) = \End_{\Sr_n}(\VV^{\ot k})$,
and $\Zs_k(\Sr_n)$  is isomorphic to $\mathsf{P}_k(n)$ when $n \geq 2k$.   Parts (a) and (c) of \cite[Thm.~5.5]{BHH} give
expressions for the dimension of  $\Zs_k^\lam(\Sr_n)$ and $\Zs_k(\Sr_n)$ respectively in terms of Stirling
numbers of the second kind, and these expressions combine with the ones above to show
 that
\begin{align}\begin{split}\label{eq:sn2}
 (n!)^{-1} \sum_{\sigma \in \Sr_n}\Fs(\sigma)^k  \overbar{ \chi_\lam(\sigma)} &=
 \dimm \Zs_k^\lam(\Sr_n) = \sum_{\ell=0}^{n}  \mathsf{K}_{\lam, (n-\ell, 1^\ell)}\,
  \binombr{k}{\ell}, \\
  (n!)^{-1} \sum_{\sigma \in \Sr_n} \mathsf{F}(\sigma)^{2k} &= \dimm \Zs_k(\Sr_n) =  \sum_{\ell=0}^n
  \binombr{2k}{\ell}. \end{split}\end{align}
 The \emph{Kostka number} $\mathsf{K}_{\lam, (n-\ell, 1^\ell)}$
 counts the number of semistandard tableaux of shape $\lam$ with $n-\ell$ entries equal to 0 and one entry
 equal to each of the numbers $1,2,\dots,\ell$ such that the entries weakly increase across the rows and
 strictly increase down the columns of the Young diagram of $\lam$ (more details on Kostka numbers can be found
 in  \cite[Sec.~2.11]{Sa} or \cite[Sec.~7.10]{S2}).
The first equality in the second  relation in \eqref{eq:sn2} was proven by Farina and  Halverson in \cite{FaH} under the additional assumption that $n\geq 2k$.  In that case, $\Zs_k(\Sr_n) \cong \mathsf{P}_k(n)$, and  the right-hand side
 $\displaystyle\sum_{\ell=0}^n  \binombr{2k}{\ell} = \sum_{\ell=0}^{2k}  \binombr{2k}{\ell}$ equals  the Bell number $\mathsf{B}(2k)$.
The relations in \eqref{eq:sn2} hold for all $n,k\in \ZZ_{\geq 1}$.

Next we examine the particular case of the symmetric group $\Sr_4$  to illustrate the above results.  \bigskip

 \begin{subsection}{The special case of the symmetric group $\Sr_4$}\end{subsection}
The irreducible modules and conjugacy classes for the symmetric group
$\Sr_4$ are indexed by the partitions $\lam \vdash 4$, \  where
$\lam \in\{(4), (3,1), (2^2), (2,1^2), (1^4)\}$. The trivial module corresponds to the partition $(4)$ with just one part,
and the 4-dimensional permutation module for $\Sr_4$ is given by $\VV = (\Sr_4)_{(4)} \oplus
(\Sr_4)_{(3,1)}$.   The corresponding McKay quiver $\cR_{\VV}(\Sr_4)$ is pictured
 in Figure \ref{fig:Rs4}.
%%%% Page 1:  S4-RepGraph%%%%%%%%%%%%%%%%%%%%%%%%%%%%%%%%%%%%%%%%%
\begin{figure}[ht!]
$$
 \begin{tikzpicture}[>=latex,text height=.5ex,text depth=0.25ex]
% Graph
\draw (2.2,-.05) node (V0){(4)};
\draw(2.3,0) node (V00){};
\draw(4, -.05) node (V1){(3,1)};
\draw(4.3,-.04)node(Va){};
\draw( 4, -.04)node(Vab){};
\draw(3.75,0) node(V11){};
 \draw (4.0,.05) node (V12){};
  \draw (4.25,0) node (V13){};
    \draw (5.0,2) node (V2){$(2^2)$};
  \draw (5,1.98) node (V22){};
    \draw (5,1.98) node (V23){};
      \draw (6,.05) node (V32){};
      \draw(3.85,-.04)node(Vaab){};
        \draw (4.9,2.1) node (Vd){};
\draw( 2.1, -.04)node(Vbc){};
%  \draw  (V2)  node[black,below=.3cm]{\textbf{\includegraphics[width=.9cm,page=3]{4-parts.pdf}}};
    \draw (6,-.05) node (V3){$(2,1^2)$};
     \draw(6.1,-.01)node(Vb){};
       \draw(6.25,-.01)node(Vbe){};
      \draw (5.65,0) node (V33){};
      \draw (6.4,0) node (V34){};
      \draw (8,-.05) node (V4){$(1^4)$};
       \draw (8.1,-.05) node (Ve){};
          \draw (7.8,0) node (V41){};
% Edges
    \path
  (V00) edge[thick]  (V11)
 (V12) edge[thick] (V22)
(V13) edge[thick](V33)
(V23) edge[thick](V32)
(V34) edge[thick](V41)
(Vd) edge [in=215,out=255,loop,thick] (Vd)
(Vab) edge [in=273,out=215,loop,thick] (Vab)
(Vaab) edge [in=273,out=215,loop,thick] (Vab)
(Vbc) edge [in=273,out=215,loop,thick] (Vbc)
(Vbe) edge [in=215,out=273,loop,thick] (Vbe)
(Vb) edge [in=215,out=272,loop,thick] (Vb)
(Ve) edge [in=215,out=272,loop,thick] (Ve)    ;
 \end{tikzpicture}
$$
\caption{McKay quiver  $\cR_\VV(\Sr_4)$ for $\VV= (\Sr_4)_{(4)} \oplus (\Sr_4)_{(3,1)}$} \label{fig:Rs4}
\end{figure}
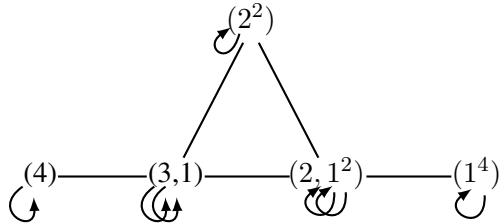
 Hence, by \eqref{eq:walkcount1},
 the dimensions of the irreducible modules
$\Zs_k^\lam(\Sr_4)$ for the centralizer algebra $\Zs_k(\Sr_4) = \End_{\Sr_4}(\VV^{\ot k})$ are given by
$$\dimm \Zs_k^\lam(\Sr_4) = (\A^k)_{(4),\lam} = ( 24)^{-1}\sum_{\mu \vdash 4} \vert \mathcal C_\mu\vert \
\chi_{{}_{\VV}}(\cm_\mu)^k \overbar{\chi_{\lam}(\cm_\mu)}.$$
The necessary information to evaluate  this expression is displayed in the table below and can be gotten
from the character table for $\Sr_4$ (see for example \cite[Sec.~2.3]{FuH}).  \begin{equation}\label{eq:tabS4}
\begin{tabular}[t]{|c||c|c|c|c|c|c|}
\hline
        $\lam\setminus\mu$
       & $(1^4)$ & $(2,1^2)$ & $(2^2)$ &
       $(3,1)$ & $(4)$ 	\\
	\hline \hline
      $|\mathcal C_\mu|$ &
      1&\ 6& \  3 & \ 8 & \ 6 \\
      \hline
       $\chi_{(4)}(\cm_\mu)$
       & $1$ & \ $1$ & \  $1$ &
      \  $1$ & \ $1$  	\\
	\hline
	 $\chi_{(3,1)}(\cm_\mu)$
       & $3$ & \ $1$ & $-1$ &
       \ $0$ & $-1$ \\
	\hline
	 $\chi_{(2^2)}(\cm_\mu)$
       & $2$ & $0$ & \  $2$ &
       $-1$ &
      \  $0$
	\\
	\hline
	 $\chi_{(2,1^2)}(\cm_\mu)$
       & $3$ & $-1$ & $-1$ &
      \  $0$ & \ $1$  	\\
	\hline
	 $\chi_{(1^4)}(\cm_\mu)$
       & $1$ & $-1$ & \ $1$ &
      \  $1$ & $-1$
	\\
	\hline
	$\chi_{{}_\VV}^k(\cm_\mu)$
       & $4^k$ & \ $2^k$ &\  $0$ &
      \  $1$ & \ $0$
	\\ \hline
\end{tabular}
\end{equation}
From this we determine that for $k \geq 1$,
\begin{align}\begin{split}\label{eq:S4dims}  & \dimm \Zs_k^{(4)}(\Sr_4) =  \frac{1}{24}\left(4^k + 6\cdot 2^k + 8\right) \ \  \left (=\sum_{\ell=1}^4\binombr{k}{\ell} \right) \\
 & \dimm \Zs_k^{(3,1)}(\Sr_4)   =   \frac{1}{24}\left(3 \cdot 4^k + 6\cdot 2^k\right) \ \ \left (= \binombr{k}{ 1} +2\binombr{k}{2}+3\binombr{k}{3}+3\binombr{k}{4} \right) \\
  & \dimm \Zs_k^{(2^2)}(\Sr_4) =   \frac{1}{24}\left(2 \cdot 4^k -8\right) \ \  \left (= \binombr{k}{2}+2\binombr{k}{3}+2\binombr{k}{4}\right) \\
    &  \dimm \Zs_k^{(2,1^2)}(\Sr_4)  =   \frac{1}{24}\left(3 \cdot 4^k - 6\cdot 2^k\right) \ \  \left (=
    \binombr{k}{2}+3\binombr{k}{3}+3\binombr{k}{4}\right) \\
  &  \dimm \Zs_k^{(1^4)}(\Sr_4)  =  \frac{1}{24}\left(4^k - 6\cdot 2^k+8\right) \ \  \left (= \binombr{k}{3}+ \binombr{k}{4} \right) \\
  & \dimm \Zs_k(\Sr_4)  = \dimm \Zs_{2k}^{(4)}(\Sr_4)  =  \frac{1}{24}\left(4^{2k} + 6\cdot 2^{2k} + 8\right) \ \
 \left (=\sum_{\ell=1}^4\binombr{2k}{\ell} \right). \end{split}\end{align}
 On the right-hand side above, we have given expressions for the dimensions
in terms of Stirling numbers of the second kind, which were derived using
the following closed-form formula:
\begin{equation}\label{eq:Stir3}
 \binombr{k}{\ell}  =  \frac{1}{\ell\,!} \sum_{j=0}^\ell (-1)^{\ell-j} \binom{\ell}{j}  j^k.\end{equation}
The coefficients of the Stirling numbers $\binombr{k}{\ell}$ are the Kostka numbers $\mathsf{K}_{\lam,(n-\ell, 1^\ell)}$ for $n=4$,  and they enumerate
the semistandard tableaux of shape $\lam$ and type $(4-\ell,1^\ell)$ as pictured below for $\lam = (2^2)$:
\medskip

\begin{center}{$ \begin{array}{c}
{\begin{tikzpicture}[scale=.6,line width=8pt]
\tikzstyle{Pedge} = [draw,line width=.7pt,-,black]
\foreach \i in {1,...,17}
{\path (\i,1) coordinate (T\i);
\path (\i,0) coordinate (B\i);
\path (\i,-1) coordinate (U\i);
\foreach \i in {1,...,17}  \path (\i,-2) coordinate (S\i);}
\path (T1) edge[Pedge] (T3);
\path(T1) edge[Pedge] (U1);
\path(T3) edge[Pedge] (U3);
\path (B1) edge[Pedge] (B3);
\path(T2) edge[Pedge] (U2);
\path (U1) edge[Pedge] (U3);
\draw  (S2)  node[black]{\small $\ell = 2$};

\node at (1.5,0.5) {$0$} ;
\node at (2.5,0.5) {$0$} ;
\node at (1.5,-.5) {$1$} ;
\node at (2.5,-.5) {$2$} ;

\path (T5) edge[Pedge] (T7);
\path(T5) edge[Pedge] (U5);
\path(T7) edge[Pedge] (U7);
\path (B5) edge[Pedge] (B7);
\path(T6) edge[Pedge] (U6);
\path (U5) edge[Pedge] (U7);
\path (T8) edge[Pedge] (T10);
\path(T8) edge[Pedge] (U8);
\path(T10) edge[Pedge] (U10);
\path (B8) edge[Pedge] (B10);
\path(T9) edge[Pedge] (U9);
\path (U8) edge[Pedge] (U10);

\node at (5.5,0.5) {$0$} ;
\node at (6.5,0.5) {$1$} ;
\node at (5.5,-.5) {$2$} ;
\node at (6.5,-.5) {$3$} ;

\node at (8.5,0.5) {$0$} ;
\node at (9.5,0.5) {$2$} ;
\node at (8.5,-.5) {$1$} ;
\node at (9.5,-.5) {$3$} ;

\draw  (S7)  node[black]{\small \, \ $\ell \, = \, 3$};
%\draw (S8)  node[black]{\small $3$};

\path (T12) edge[Pedge] (T14);
\path(T12) edge[Pedge] (U12);
\path(T14) edge[Pedge] (U14);
\path (B12) edge[Pedge] (B14);
\path(T13) edge[Pedge] (U13);
\path (U12) edge[Pedge] (U14);
\path (T15) edge[Pedge] (T17);
\path(T15) edge[Pedge] (U15);
\path(T17) edge[Pedge] (U17);
\path (B15) edge[Pedge] (B17);
\path(T16) edge[Pedge] (U16);
\path (U15) edge[Pedge] (U17);

\node at (12.5,0.5) {$1$} ;
\node at (13.5,0.5) {$2$} ;
\node at (12.5,-.5) {$3$} ;
\node at (13.5,-.5) {$4$} ;

\node at (15.5,0.5) {$1$} ;
\node at (16.5,0.5) {$3$} ;
\node at (15.5,-.5) {$2$} ;
\node at (16.5,-.5) {$4$} ;

\draw  (S14)  node[black]{\small \, \ $\ell \, = \, 4$};

%%\draw (S15)  node[black]{\small $4$};
%%%%%%%
\end{tikzpicture}}\end{array}$.}\end{center}

\begin{subsection}{Bratteli diagram}\label{S:Brat} \end{subsection}

The \emph{Bratteli diagram} $\mathcal{B}_{\VV}(\GG)$ is an
infinite graph  with vertices labeled by the elements of  $\Lambda_k(\GG)$ on level $k$.      A walk of $k$ steps on the McKay quiver  $\cR_{\VV}(\GG)$ from $\zr$ to $\lambda$ is a sequence \newline $\left(\lam^{(0)} = 0, \lambda^{(1)}, \lambda^{(2)}, \ldots, \lambda^{(k)}=\lambda\right)$  starting at $\lambda^{(0)} = \zr$,   such that  $\lambda^{(j)} \in \Lambda_j(\GG)$   for each $1 \le j \le k$, and  $\lambda^{(j-1)}$ is connected to  ${\lambda^{(j)}}$ by an edge in $\cR_{\VV}(\GG)$.   Such a walk is equivalent to a unique path of length $k$ on the Bratteli diagram $\mathcal{B}_{\VV}(\GG)$ from $\zr$
 at the top  to $\lambda \in \Lambda_k(\GG)$ on level $k$.   The subscript on vertex  $\lambda \in \Lambda_k(\GG)$ in $\mathcal{B}_\VV(\GG)$  indicates  the number  $\mathsf{m}_k^\lambda$  of paths from $\zr$ on the top  to $\lambda$ at level $k$.  This can be easily  computed by summing, in a Pascal triangle fashion, the subscripts of the vertices at level $k-1$ that are connected to $\lambda$.  This is
dimension of the irreducible $\Zs_k(\GG)$-module  $\Zs_k^\lam(\GG)$, which is also the multiplicity of
$\GG_\lam$ in $\VV^{\ot k}$.  The sum of the squares of
those dimensions at level $k$ is the number on the right, which is the dimension of the centralizer algebra $ \Zs_k(\GG)$
by Wedderburn theory. The subscripts in the columns of $\mathcal{B}_\VV(\GG)$ are given by the formulas  in \eqref{eq:S4dims}.

The top levels of the
 Bratteli diagram for  the group $\GG = \Sr_4$ and its 4-dimensional permutation module $\VV$  are exhibited in Figure \ref{fig:S4Bratteli}.
 \begin{figure}[ht!]
$$
\begin{array}{c} \includegraphics[scale=1,page=1]{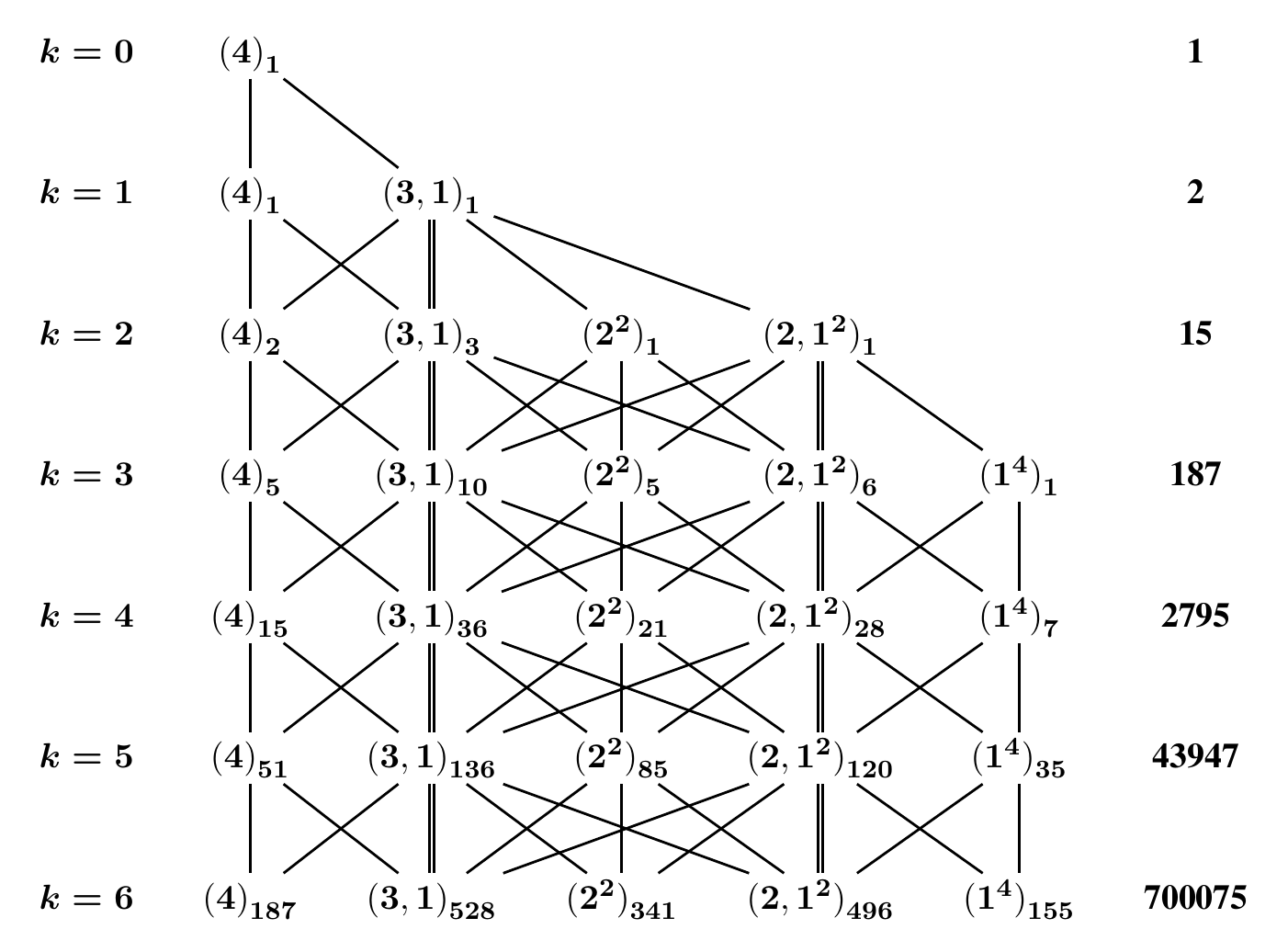} \end{array}
$$
\caption{Levels $k = 0,1, \dots, 6$  of the Bratteli diagram $\mathcal{B}_\VV(\Sr_4)$ for $\Sr_4$ and its permutation module $\VV$}\label{fig:S4Bratteli}
\end{figure}

 \bigskip

\begin{subsection}{The group $\ZZ_r \wr \Sr_n$} \end{subsection}
In this section,  $\GG$ is the wreath product  $\ZZ_r \wr \Sr_n$ viewed as $n \times n$ monomial matrices with entries of the form $\om^j$ for $j = 0,1,\dots, r-1$,
  where $\om = \er^{2\pi i/r}$, a primitive $r$th root of unity for $r \geq 2$.   The module  $\VV$ is the space of $n \times 1$
column vectors with complex entries
  on which $\GG$ acts by matrix multiplication.    Our main result (Theorem \ref{T:wreath} below) is a formula for the dimension of  the
  $\GG$-invariants  $(\VV^{\ot k})^\GG$ in $\VV^{\ot k}$, equivalently, for the dimension $\dimm \Zs_k^\zr(\GG)
  = \vert \GG \vert^{-1} \sum_{g \in \GG} \chi_{{}_\VV}(g)^k$ of the irreducible module labeled by $\zr$ for the centralizer
  algebra $\Zs_k(\GG) = \End_{\GG}(\VV^{\ot k})$.  Our formula will depend on the number of
  entries on the main diagonal of a monomial matrix in $\GG$ (the number of fixed points of the underlying permutation
  in $\Sr_n$), and so for $m = 1,2,\ldots,n$,  we set $\Fs_n(m)  := \vert \{ \sigma \in \Sr_n  \mid   \Fs(\sigma) = m\} \vert.$    This number, which
  is sometimes referred to as a \emph{rencontres number}, counts  the number of  ``partial derangements''  of $n$ with $m$ fixed points.
 It  equals $\binom{n}{m} \mathsf{D}_{n-m}$,  where $\mathsf{D}_{n-m}$ is the number of \emph{derangements}
  of $n-m$ (permutations in $\Sr_{n-m}$ with no fixed points).    From known expressions for the derangement numbers, we
  have
  \begin{equation}\label{eq:derange}  \Fs_n(m) = \binom{n}{m} \mathsf{D}_{n-m}  =  \binom{n}{m} (n-m)! \sum_{j=0}^{n-m} \frac{(-1)^j}{ j\,!}   =  \frac{n!}{m!} \sum_{j=0}^{n-m} \frac{(-1)^j}{ j\,!}. \end{equation}
\medskip

 \begin{thm}\label{T:wreath} Assume  $\GG = \ZZ_r \wr \Sr_n$ and  $\VV$ is the  $n$-dimensional $\GG$-module on which $\GG$ acts by
 monomial matrices.  Then the following hold:
 \begin{itemize} \item[{\rm (a)}] The dimension of the space of $\GG$-invariants in $\VV^{\ot k}$ (equivalently, $\dimm \Zs_k^\zr(\GG)$)  is given by
 \begin{equation}\label{eq:wreath}  \dimm (\VV^{\ot k})^\GG = \frac{1}{r^nn\,!} \sum_{m=1}^n  r^m \Fs_n(m)^k
\left( \sum_{\ell_1,\ell_2,\ldots, \ell_m}  \binom{k}{\ell_1,\ell_2, \dots, \ell_m}\right),   \end{equation}
where  the sum is over all $0 \leq \ell_1,\ell_2, \dots, \ell_m\leq k$ such that $\ell_1 +\ell_2+ \cdots + \ell_m = k$ and $\ell_1 \equiv \ell_2 \equiv  \cdots
  \equiv \ell_m \equiv 0  \, \modd r$, and
$ \Fs_n(m) = \displaystyle{\frac{n!}{m!} \sum_{j=0}^{n-m} \frac{(-1)^j}{ j\,!}}$.    In particular, the space $(\VV^{\ot k})^\GG$ of invariants is 0 unless $k \equiv 0 \modd r$.
\item[{\rm(b)}]  The exponential generating function for the $\GG$-invariants is
 \begin{equation} \gr^\zr(t)= \sum_{k=0}^\infty    \dimm (\VV^{\ot k})^\GG  \frac{t^k}{k!}
 =  \frac{1}{r^nn\,!} \sum_{m=1}^n  r^{m}\ \hr_1\left(\Fs_n(m)t,r\right)^m, \end{equation}
where  $\hr_1(t,r) = \displaystyle{\sum_{q=0}^\infty  \frac{t^{qr}}{(qr)!}}$ is a generalized hyperbolic function.
\end{itemize} \end{thm}

 \begin{proof} (a) \ We know from Theorem \ref{T:walk} that
 $\dimm (\VV^{\ot k})^\GG = (\A^k)_{\zr,\zr} = \vert \GG \vert^{-1} \sum_{g \in \GG}  \chi_{{}_\VV}(g)^k$, from which we have

\begin{align*}  \dimm (\VV^{\ot k})^\GG  &= \frac{1}{r^n \, n\,!}\sum_{m=1}^n \Fs_n(m)^k  \sum_{b_1,b_2,\ldots, b_m \in\{0,1,\dots,r-1\}}  \left(\om^{b_1}+\om^{b_2} +\cdots + \om^{b_m}\right)^k  \\
 &\hspace{-1.2cm} = \frac{1}{r^n \, n\,!}\sum_{m=1}^n \Fs_n(m)^k \hspace{-.1cm}\left( \sum_{\ell_1+\ell_2+\dots + \ell_m = k} \binom{k}{\ell_1,\ell_2, \dots,\ell_m}
\left( \sum_{b_1 =0}^{r-1}  \om^{\ell_1 b_1}\right)\left( \sum_{b_2 =0}^{r-1}  \om^{\ell_2 b_2}\right) \cdots \left( \sum_{b_m=0}^{r-1} \om^{\ell_m b_m}\right)\right) \\
 &\hspace{-1.2cm} =
 \frac{1}{r^n\,n\,!}\sum_{m=1}^n \Fs_n(m)^k r^m \left( \sum_{\substack{\ell_1+\ell_2+\cdots + \ell_m = k \\
 \ell_1 \equiv\ell_2 \equiv \cdots \equiv \ell_m \equiv 0 \modd r}} \binom{k}{\ell_1,\ell_2, \dots, \ell_m} \right)   \hspace{1.2cm}  \text{by \eqref{eq:omsum}}.
 \end{align*}

(b) \ It is a consequence of \eqref{eq:wreath}  that for $\GG = \ZZ_r \wr \Sr_n$,
\begin{equation}
\dimm (\VV^{\ot k})^\GG = \frac{1}{r^nn\,!} \sum_{m=1}^n  r^m \Fs_n(m)^k\left( \sum_{(q_1+q_2+\cdots+q_m)r = k}  \binom{k}{q_1r,q_2r, \dots, q_mr}\right).
\end{equation}
Therefore,  the exponential generating function for the invariants is given by
\begin{align}\label{eq:genhypo}  \gr^\zr(t)&= \sum_{k=0}^\infty    \dimm (\VV^{\ot k})^\GG  \frac{t^k}{k!}  =
 \frac{1}{r^nn\,!} \sum_{m=1}^n  r^m  \sum_{k=0}^\infty \Fs_n(m)^k \left(\sum_{(q_1+q_2+\cdots+q_m)r = k}  \binom{k}{q_1r,q_2r, \dots, q_mr} \frac{t^k}{k!}\right) \nonumber  \\
 &\hspace{-.2cm}=  \frac{1}{r^nn\,!} \sum_{m=1}^n  r^{m} \left( \sum_{q_1=0}^\infty \frac{\left(\Fs_n(m)t\right)^{q_1r}} {(q_1r)!}\right)
\left( \sum_{q_2=0}^\infty \frac{\left(\Fs_n(m)t\right)^{q_2r}} {(q_2r)!}\right) \cdots \left(\sum_{q_m=0}^\infty \frac{\left(\Fs_n(m)t\right)^{q_m r}} {(q_mr)!}\right) \nonumber  \\
&\hspace{-.2cm}=  \frac{1}{r^nn\,!} \sum_{m=1}^n r^m \Big(\hr_1\left(\Fs_n(m)t,r\right)\Big)^m,   \end{align}
where  $\hr_1(t,r)$ is the generalized hyperbolic function $\hr_1(t,r) = \displaystyle{\sum_{q=0}^\infty  \frac{t^{qr}}{(qr)!}}$ \,  (see \eqref{eq:h1} and \eqref{eq:powerser} below for more details).  \end{proof}

 \begin{subsection}{$\GG = \ZZ_r \wr \Sr_n$ for some special choices of $r$ and $n$} \end{subsection}

 Assume  $\GG = \ZZ_r \wr \Sr_2$ and $\VV = \CC^2$.    Then since $\Fs_2(1) = \binom{2}{1} \mathsf{D}_1 = 0$,
 and  $\Fs_2(2) = \binom{2}{2} \mathsf{D}_0 = 1$,  we have
\begin{equation}\dimm (\VV^{\ot k})^\GG = \dimm \Zs_k^{\zr}(\GG) =  \half  \sum_{\substack{\ell_1+ \ell_2= k \\  \ell_1 \equiv \ell_2 \equiv 0 \modd r}} \binom{k}{\ell_1, \ell_2}.
\end{equation}
So, for example, when $r = 2$,
\begin{align}\begin{split} \label{eq:z2s2} \dimm (\VV^{\ot k})^\GG &= \begin{cases}  \displaystyle{\half\sum_{\ell=0}^{\half k} \binom{k}{{2\ell}}} =  \half 2^{k-1} = 2^{k-2} & \qquad \text{if \, $k$ \, is even and \, $k \geq 2$,}  \\
0 & \qquad \text{if \, $k$ \, is odd and \, $k \geq 1$.}  \end{cases} \\
\mathsf{P}^\zr(t) &= \sum_{k=0}^\infty  \dimm (\VV^{\ot k})^\GG\,  t^k =   1+t^2 \sum_{j=0}^\infty  (4t^2)^j  =
1 + \frac{t^2}{1-4t^2} =  \frac{1-3t^2}{1-4t^2}.  \end{split}\end{align}

 \begin{subsection}{The group $\GG = \ZZ_r \wr \Sr_n$ -- a different approach} \end{subsection}
The irreducible modules $\GG_{\bal}$  for $\GG= \ZZ_r \wr \Sr_n$, hence also the $\GG$-conjugacy classes $\mathcal C_{\bal}$,  are labeled
by $r$-tuples of partitions  ${\bal} = (\alpha^{(1)}, \alpha^{(2)}, \dots,\alpha^{(r)})$
such that $n = \sum_{i=1}^r |\alpha^{(i)}|$
(see for example \cite[Sec.~2]{AK}).

 For $x \in \CC$,  let  $\mathsf{J}_\ell(x)$ be the $\ell \times \ell$ Jordan block matrix given by
 $$\mathsf{J}_\ell(x) = \left( \begin{matrix} 0 & 1 &  & & & & & \\ 0 & 0 & 1  & & & & &  \\
 && &  \ddots &&&&&\\ \vdots & & & & & & & \\   &&&& \ddots &&& \\ & & & & & & 0 &1 \\  x & 0 & &&&& 0 &  0 \end{matrix} \right)$$
 Then a conjugacy class representative of $\GG$ corresponding to $\bal$ is
 $$\mathsf{c}_{\bal}  =  \bigoplus_{i=1}^r   \bigoplus_{p}    \mathsf{J}_{\alpha_p^{(i)}} (\om^{i-1}),$$
 where $\om = \er^{2\pi i/r}$,  the parts $\alpha_p^{(i)}$ of the $i$th partition  $\alpha^{(i)}$ are  $\alpha_1^{(i)} \geq \alpha_2^{(i)} \geq \ldots$,
 and this sum represents the $n \times n$ matrix with blocks down the main diagonal starting
 with $ \mathsf{J}_{\alpha_1^{(1)}}(\om^0)$, then $ \mathsf{J}_{\alpha_2^{(1)}}(\om^0)$,  $\ldots$, and continuing
 down to $ \mathsf{J}_{\alpha_\ell^{(r)}}(\om^{r-1})$ corresponding to the last part
$\alpha_\ell^{(r)}$ of the last partition $\alpha^{(r)}$.

For  a partition $\lam$,   assume $\mathsf{p}_j(\lambda)$ is the number
of parts of $\lambda$ equal to $j$.        Set
$$\mathsf{z}_\lambda = \prod_{j =1}^n  j^{\mathsf{p}_j(\lambda)} \, \mathsf{p}_j(\lambda)!.$$
This is the order of the centralizer of an element of $\Sr_{|\lambda|}$ with cycle structure given by the partition $\lambda$.
\medskip
Now for $\bal = (\alpha^{(1)},\alpha^{(2)}, \ldots, \alpha^{(r)})$,  we define
\begin{equation}\label{eq:multdefs} \mathsf{p}_j(\bal) =
\sum_{i=1}^r \mathsf{p}_j(\alpha^{(i)})   \qquad \hbox{\rm and} \qquad \mathsf{p}(\bal) = \sum_{j=1}^n  \mathsf{p}_j(\bal).
\end{equation}
Thus, $\mathsf{p}_j(\bal)$ is the total number of parts equal to $j$ in the partitions comprising $\bal$, and
$\mathsf{p}(\bal)$ is the total number of nonzero parts in the partitions of $\bal$.
Then according to  \cite[Sec.~2]{AK}, the size of the centralizer of $\mathsf{c}_{\bal}$ in $\GG$ is given by
\begin{equation}\label{eq:cent2}
\mathsf{z}_{\bal}  =  \prod_{i,j}  (r j)^{\mathsf{p}_j(\alpha^{(i)})}\,  \mathsf{p}_j(\alpha^{(i)})!  =
r^{\mathsf{p}(\bal)} \prod_{j=1}^n   j^{\mathsf{p}_j(\bal)}\, \left(\prod_{i=1}^r  \mathsf{p}_j(\alpha^{(i)}) ! \right)=
r^{\mathsf{p}(\bal)} \prod_{i=1}^r    \mathsf{z}_{\alpha^{(i)}}.
\end{equation}
Hence,  the size of the conjugacy class $\mathcal{C}_{\bal}$ corresponding to the element
$\mathsf{c}_{\bal}$  is  given by
$$|\mathcal{C}_{\bal}| =  \frac{ | \GG|}{\mathsf{z}_{\bal}}   =  \frac{ | \GG|}
{r^{\mathsf{p}(\bal)} \,\prod_{j=1}^n    j^{\mathsf{p}_{j}(\bal) } \, \left(\prod_{i=1}^r  \mathsf{p}_{j}(\alpha^{(i)})!\right)}.  $$
Thus,  we know that
$$\dimm (\VV^{\ot k})^\GG = \dimm \Zs_k^{\zr}(\GG)  =   | \GG|^{-1}   \sum_{\bal \in \Lambda(\GG)} |\mathcal{C}_{\bal}|\, \chi_{{}_\VV}(\cm_{\bal})^{k} =
\sum_{{\bal}  \in \Lambda(\GG)}
\frac{ \chi_{{}_\VV}(\mathsf{c}_{\bal})^{k} }
{r^{\mathsf{p}(\bal)} \,\prod_{j =1}^n    j^{\mathsf{p}_{j}(\bal) } \, \left(\prod_{i=1}^r  \mathsf{p}_{j}(\alpha^{(i)})!\right)}.$$
Observe that
$$\chi_{{}_\VV}(\cm_{\bal}) = \mathsf{tr}_{\VV}\left(\cm_{\bal}\right)
=   \sum_{i=1}^r  \mathsf{p}_{1}(\alpha^{(i)}) \,\om^{i-1}  =   \sum_{i=1}^r  \mathsf{F}(\alpha^{(i)}) \,\om^{i-1} $$
where $\mathsf{p}_1(\alpha^{(i)})$ is the number of parts equal to 1 in $\alpha^{(i)}$, as the only contributions to the trace come from the matrix blocks of size one in
$\cm_{\bal}$.  Since that is
the number of fixed points of a permutation of cycle type $\alpha^{(i)}$, we write
$ \mathsf{F}(\alpha^{(i)})$ by a slight abuse of notation.
Therefore, we obtain a second expression for the dimension of the $\GG$-invariants in $\VV^{\ot k}$
using the definitions in \eqref{eq:multdefs}:
\begin{equation}\label{eq:inv1}\dimm (\VV^{\ot k})^\GG  =  \dimm \Zs_k^{\zr}(\GG)
= \sum_{\bal \in \Lambda(\GG)} \, \,  \frac{\Big(\sum_{i=1}^r  \mathsf{F}(\alpha^{(i)})\, \om^{i-1}\Big)^k}
{r^{\mathsf{p}(\bal)} \,\prod_{j =1}^n    j^{\mathsf{p}_{j}(\bal) } \, \left(\prod_{i=1}^r  \mathsf{p}_{j}(\alpha^{(i)})!\right)} \quad \hbox{\rm for $\GG = \ZZ_r \wr \Sr_n$},\end{equation}

\subsection{The group $\GG = \ZZ_2 \wr \Sr_n$}
The group $\GG = \ZZ_2 \wr \Sr_n$ is the Weyl group for a root system of type $\mathrm B_n$  or  $\mathrm C_n$.
The irreducible $\GG$-modules are labeled by pairs $\bal = (\alpha^{(1)}, \alpha^{(2)})$ of partitions such that
$|\alpha^{(1)}| + |\alpha^{(2)}| = n$.    Since $\om = -1$ in this case,  we have the following formula for the dimension of the space
of $\GG$-invariants in $\VV^{\ot k}$ in part (a) of the next Theorem.  Part (b) is an immediate consequence of Theorem
\ref{T:wreath}\,(b).  \medskip

\begin{thm}\label{T:WgroupBC}  Assume $\GG = \ZZ_2 \wr \Sr_n$ is the Weyl group for a root system of type $\mathrm B_n$  or  $\mathrm C_n$
and $\VV$ is the $n$-dimensional  $\GG$-module on which $\GG$ acts by monomial matrices with entries $\pm 1$.   Then
\begin{itemize}\item[{\rm(a)}]
\begin{equation} \label{eq:z2sn}\dimm (\VV^{\ot k})^\GG  =  \dimm \Zs_k^{\zr}(\GG)
= \sum_{\bal \in \Lambda(\GG)} \, \,  \frac{\Big(\mathsf{F}(\alpha^{(1)}) - \mathsf{F}(\alpha^{(2)})\Big)^k}
{2^{\mathsf{p}(\bal)} \,\prod_{j =1}^n    j^{\mathsf{p}_{j}(\bal) } \, \left(\mathsf{p}_{j}(\alpha^{(1)})! \cdot \mathsf{p}_{j}(\alpha^{(2)})! \right)}, \end{equation}
where $\bal$ ranges over pairs $\bal = (\alpha^{(1)}, \alpha^{(2)})$ of partitions such that
$|\alpha^{(1)}| + |\alpha^{(2)}| = n$; \  $\mathsf{p}_j(\alpha^{(i)})$ and $\mathsf{p}(\bal)$ are as in \eqref{eq:multdefs};  and $ \mathsf{F}(\alpha^{(i)})$ is the number of fixed points of a permutation of cycle type $\alpha^{(i)}$.
\item[{\rm (b)}]  The exponential generating function for the $\GG$-invariants is
 \begin{equation} \gr^\zr(t)= \sum_{k=0}^\infty    \dimm (\VV^{\ot k})^\GG  \frac{t^k}{k!}
 =  \frac{1}{2^nn\,!} \sum_{m=1}^n \ 2^m \Big(\mathsf{cosh}\left(\Fs_n(m)t\right)\Big)^m. \end{equation}
 \end{itemize}
\end{thm}
\medskip

\begin{remark}{\rm   In \cite{T}, Tanabe investigated the centralizer algebra $\Zs_k(\GG)$, where $\GG$ is a complex reflection group
$\GG(m,p,n)$ viewed as $n\times n$ matrices acting on $\VV = \CC^n$.  The group $\GG(r,1,n)$ is the wreath
product  $\ZZ_r \wr \Sr_n$.  Using results from \cite{T}, we showed in \cite{BM} for $\GG = \ZZ_2 \wr \Sr_n$ that
$$\dimm \Zs_k(\GG) = \sum_{s=1}^n   \mathrm{T}(k, s),$$
where $\mathrm{T}(k,s)$ is the number of set partitions of a set of size $2k$ into $s$ nonempty disjoint parts of \emph{even} size.
The numbers  $\mathrm{T}(k,s)$ correspond to sequence A156289 in the Online Encyclopedia of Integer Sequences
[OEIS] and have many different interpretations.  They  are known to satisfy
$$\mathrm{T}(k, s) = \frac{1}{s! \,2^{s-1}}  \sum_{j=1}^s  (-1)^{s-j}  \binom{2s}{s-j}  j^{2k} = \sum_\lam
\frac{1}{\prod_{j \geq 1} \mathsf{p}_j(\lam)}  \binom{2k}{2\lam_1, 2 \lam_2, \, \ldots, 2 \lam_s},$$
where the last sum is over all partitions  $\lam = \{\lam_1 \geq \lam_2 \geq \dots \geq \lam_s > 0\}$ of $k$
into $s$ nonzero parts $\lam_i$  (see \cite[Sec.~4.2]{BM} for details).   In particular,  since $\VV$ is self-dual, we see that
\begin{equation}\label{eq:z22k}  \dimm (\VV^{\ot 2k})^\GG =  \dimm \Zs_k(\GG) =  \sum_{s=1}^n   \mathrm{T}(k, s),   \qquad \hbox{\rm for $\GG = \ZZ_2 \wr \Sr_n$}.\end{equation}
It would be interesting to show the equivalence of the formulas in Theorem \ref{T:wreath} and  \eqref{eq:z2sn} and then relate them
(with $2k$ in place of $k$)  to  \eqref{eq:z22k}.}  \end{remark}

%%%%%%%%%%%%%%%%%%%%%%%%%%%%%%%%%%%%%%%%%%%%%%%%%%

\section{$\GG = \GL_2(\FF_q)$ and $\GG = \SL_2(\FF_q)$}\label{S:linear groups}
Let $\FF_q$ be a finite field of $q$ elements.  Then $q = p^\ell$ for some prime $p$ and some  $\ell \geq 1$, and we
assume  $p$ is odd to simplify considerations.
In this section,  $\GG$ is the general linear group $\GL_2(\FF_q)$ of $2 \times 2$ invertible matrices over $\FF_q$
or  the special linear subgroup   $\SL_2(\FF_q)$ of  matrices with determinant equal to 1.    We assume $\VV = \mathsf{Ind}_{\BB}^\GG  \BB_\zr$, the
$\GG$-module induced from the trivial module $\BB_\zr$  for the subgroup $\BB$ of  upper triangular matrices in $\GG$,
and $\VV_q$ is its $q$-dimensional irreducible summand, which is  Steinberg module.  (Here we write $\VV_q$ rather
than the customary $\mathsf{St}$,  to emphasize its analogy to  $\VV$ in previous sections.)  Our aim in this section
is to develop a formula for $\dimm (\VV^{\ot k})^\GG$ and for $\dimm (\VV_q^{\ot k})^\GG$ and to determine the
corresponding Poincar\'e series for the tensor invariants.    \medskip

\begin{subsection}{$\GG = \GL_2(\FF_q)$} \end{subsection}
Let  $\BB  = \left \{\left(\begin{matrix} x &  y\\ 0 & \, z \end{matrix}\right)\, \Big | \,  x,z \in \FF_q^\times, \, y \in \FF_q\right\}$
be the Borel subgroup of upper-triangular matrices in $\GG = \GL_2(\FF_q)$  and   $\VV$ be the induced $\GG$-module  $\VV = \mathsf{Ind}_{\BB}^{\GG} \BB_\zr
= \CC[\GG] \ot_{\CC[\BB]} \BB_\zr$.   Since the order of $\GG$ is $q(q+1)(q-1)^2$ and the
order of $\mathsf{B}$ is $q(q-1)^2$, we have $\dimm \VV = q+1$.     The module $\VV$ decomposes into a sum $\VV = \GG_\zr \oplus \VV_q$ of a  copy of the trivial $\GG$-module $\GG_\zr$  and a copy of a $q$-dimensional irreducible $\GG$-module
$\VV_q$  (the Steinberg module).

Let  $\ve$ be a non-square in $\FF_q^\times$,  and define the following elements of $\GG$,
\begin{equation}\begin{array}{cccc}  \mathsf{a}_x = \left(\begin{matrix} x &  0 \\ 0 & x \end{matrix}\right),  & \quad
\mathsf{b}_x = \left(\begin{matrix} x &  1  \\ 0 & x\end{matrix}\right),  & \quad
\mathsf{c}_{x,y}= \left(\begin{matrix} x &  0 \\ 0 & y \end{matrix}\right),  & \quad
\mathsf{d}_{x,y} =  \left(\begin{matrix} x & \ve y  \\  y  & x\end{matrix}\right) \\
\qquad \quad  (x \in \FF_q^\times)  \quad &\qquad \qquad \ (x \in \FF_q^\times)  \quad & \quad \ (x,y \in \FF_q^\times, x \neq y)
 \quad & \qquad \qquad \ (y \in \FF_q^\times)  \end{array} \end{equation}
We will use the information in the table below, which can
be derived from \cite[Sec.~5.2]{FuH}.   As before,  $\mathsf{c}_\mu$, $\mu \in \Lambda(\GG)$,  is a representative of
the conjugacy class $\mathcal{C}_\mu$ of $\GG$.
\begin{equation}\label{tab:gl}
\begin{tabular}[t]{|c||c|c|c|c|c|}
\hline
        $\mathsf{c}_\mu$
       & $\mathsf{a}_x$ & $\mathsf{b}_x$ & $ \mathsf{c}_{x,y}$ &
       $\mathsf{d}_{x,y}$	\\
	\hline \hline
	no. of such classes
       & $q-1$ & $q-1$ & $\half(q-1)(q-2)$  & $\half q(q-1)$	\\
	\hline
      $|\mathcal C_\mu|$
       & $1$ & $q^2-1$ & $q^2+q$  & $q^2-q$	\\
      \hline
       $\chi_{{}_\VV}(\mathsf{c}_\mu)$
       & $q+1$ & \ $1$ & \  $2$ &
      \ \, $0$ 	\\  \hline
      $\chi_{{}_{\VV_q}}(\mathsf{c}_\mu)$
       & $q$ & \ $0$ & \  $1$ &
      \  $-1$ 	\\
	\hline  \end{tabular}
\end{equation}
Therefore, we have the following consequence of Theorem \ref{T:walk}. \bigskip

\begin{thm}\label{T:GL2q}  Assume $\GG = \GL_2(\FF_q)$ where $q$ is odd.
\begin{itemize} \item[{\rm (a)}]  For the
$\GG$-module  $\VV = \mathsf{Ind}_\BB^\GG \BB_\zr = \GG_\zr \oplus \VV_q$ induced from the trivial module $\BB_\zr$  for the Borel subgroup $\BB$ of upper-triangular matrices in $\GG$,
\begin{equation}\dimm (\VV^{\ot k})^\GG  = \begin{cases}
\quad \ \  1 &   \ \  \hbox{\rm  when\ $k=0$,}\\
\displaystyle{ \frac{1}{q(q-1)}\left((q+1)^{k-1} +q(q-2)\cdot 2^{k-1} + q-1\right)} &  \ \  \hbox{\rm when\ $k \geq 1$.} \end{cases} \end{equation}
The Poincar\'e series for the $\GG$-invariants $\mathsf{T}(\VV)^\GG$  in $\mathsf{T}(\VV) = \bigoplus_{k=0}^\infty \VV^{\ot k}$ is
\begin{equation}\mathsf{P}^0(t) = \sum_{k=0}^\infty   \dimm (\VV^{\ot k})^\GG \,  t^k
= \frac{1 - (q+3)t + (2q+3) t^2 - qt^3}{(1-t)\,(1-2t)\,(1-(1+q)t)}.  \end{equation}
\item[{\rm(b)}]  For the Steinberg module $\VV_q$,   $\dimm (\VV_q^{\ot k})^\GG = 1$ when $k = 0$, and
\begin{align}\hspace{-.5cm} \dimm (\VV_q^{\ot k})^\GG &= \frac{1}{2(q^2-1)}\left(2q^{k-1}  - q(q-1)(-1)^{k-1} + (q+1)(q-2) \right)  \quad \hbox{\rm when\ $k \geq 1$,} \\
& =\begin{cases}  \displaystyle{\frac{q^{2\ell}-1}{q^2-1}  =  \sum_{j=0}^{\ell-1}  q^{2j}}  & \qquad \hbox{\rm if \, \ $k = 2\ell + 1 \geq 1$,}\\
  \displaystyle{1+q\frac{q^{2\ell-2}-1}{q^2-1} = 1 + \sum_{j=0}^{\ell-2}  q^{2j+1}}  & \qquad \hbox{\rm if \, \ $k = 2\ell \geq 2$. }
 \end{cases} \end{align}
The Poincar\'e series $\mathsf{P}_q^0(t)$  for the $\GG$-invariants $\mathsf{T}(\VV_q)^\GG$  in $\mathsf{T}(\VV_q) = \bigoplus_{k=0}^\infty \VV_q^{\ot k}$ is
\begin{equation}\mathsf{P}_q^0(t) = \sum_{k=0}^\infty   \dimm (\VV_q^{\ot k})^\GG \,  t^k
= \frac{1 - qt + t^3}{(1-t)\,(1+t)\,(1-qt)}.
 \end{equation}
 \end{itemize}
\end{thm}

\begin{proof} (a) \,  From Theorem \ref{T:walk} and Table \ref{tab:gl} we know that
\begin{align}\begin{split}\dimm (\VV^{\ot k})^\GG &= \dimm \Zs_k^\zr(\GG) = \frac{1}{ \vert \GG \vert} \sum_{\mu \in \Lambda(\GG)}  \vert  \mathcal{C}_\mu \vert   \chi_{{}_\VV}(\mathsf{c}_\mu)^k  \nonumber  \\
& \hspace{-2cm} = \frac{1}{(q-1)^2q(q+1)}\left( (q-1)(q+1)^k  +
(q-1)(q^2-1)\, 1^k +
\half q(q+1)(q-1)(q-2)\,2^k + \half q^2(q-1)^2 \, 0^k\right) \\
& \hspace{-2cm} =  \frac{1}{q(q-1)}\left((q+1)^{k-1} +q(q-2)\cdot 2^{k-1} + q-1\right)  \quad \hbox{\rm when\ $k \geq 1$.}
\end{split}  \end{align}
  Therefore,
\begin{align*}\mathsf{P}^0(t) &= \sum_{k=0}^\infty   \dimm (\VV^{\ot k})^\GG \,  t^k
= 1 + \frac{1}{q(q-1)}\left (\sum_{k =1}^\infty (q+1)^{k-1} +q(q-2)\cdot 2^{k-1} + (q-1)\right) t^k \\
& = 1 + \frac{1}{q(q-1)}\left (t \sum_{k =1}^\infty (q+1)^{k-1}t^{k-1} + q(q-2)t  \sum_{k =1}^\infty 2^{k-1}t^{k-1}
+ (q-1)t \sum_{k = 1}^\infty  t^{k-1}\right) \\
& = 1 + \frac{1}{q(q-1)}\left (\frac{t}{1- (q+1)t}  +  \frac{q(q-2)t}{1-2t} + \frac{(q-1)t}{1-t}\right) \\
& = \frac{1 - (q+3)t + (2q+3) t^2 - q t^3}{(1-t)\,(1-2t)\,(1-(q+1)t)}.
 \end{align*}

(b) \, Now for $\VV_q$ and $k \geq 1$, we have
\begin{align}\begin{split}\dimm (\VV_q^{\ot k})^\GG &= \frac{1}{ \vert \GG \vert} \sum_{\mu \in \Lambda(\GG)}  \vert  \mathcal{C}_\mu \vert \,  \chi_{{}_{\VV_q}}(\mathsf{c}_\mu)^k  \nonumber  \\
& \hspace{-1.7cm} = \frac{1}{(q-1)^2q(q+1)}\left( (q-1)q^k  +
(q-1)(q^2-1)\, 0^k +
\half  q(q+1)(q-1)(q-2)\,1^k + \half q^2(q-1)^2 \, (-1)^k\right) \\
&\hspace{-1.7cm} =  \frac{1}{2(q^2-1)}\left(2q^{k-1} + q(q-1) (-1)^{k} +  (q+1)(q-2)\right) \\
& \hspace{-1.7cm} =\begin{cases}  \displaystyle{\frac{q^{2\ell}-1}{q^2-1}  =  \sum_{j=0}^{\ell-1}  q^{2j}}  & \qquad \hbox{\rm if \, \ $k = 2\ell + 1 \geq 1$,}\\
  \displaystyle{1+q\frac{q^{2\ell-2}-1}{q^2-1} = 1 + q \sum_{j=0}^{\ell-2}  q^{2j}}  & \qquad \hbox{\rm if \, \ $k = 2\ell \geq 2$,}
 \end{cases}
\end{split}  \end{align}
and hence,
\begin{align*}\mathsf{P}^0(t) &=
1 + \frac{1}{2(q^2-1)}\sum_{k =1}^\infty \left (2q^{k-1}+(q+1)(q-2)+q(q-1)(-1)^{k}\right) t^k \\
& = \frac{1 - qt + t^3}{(1-t)\,(1+t)\,(1-qt)}.  \end{align*}
\end{proof}

 \begin{subsection}{$\GG = \SL_2(\FF_q)$} \end{subsection}
For the group $\GG = \SL_2(\FF_q)$ ($q$ odd), we introduce the following elements of $\GG$:
\begin{equation}  \mathsf{u}_x = \left(\begin{matrix} x & 0  \\ 0 & x^{-1}\end{matrix}\right) \  (x \neq 0),   \quad
\mathsf{v}_y = \left(\begin{matrix} 1 &  y \\ 0 & 1\end{matrix}\right), \quad
\mathsf{w}_{x,y} =  \left(\begin{matrix} x & y  \\  y\ve & x \end{matrix}\right) \ (x^2 -\ve y^2 =1).  \end{equation}
We will use the information in the following table, which can
be derived from \cite[Chap.~3]{Mur} or \cite[Sec.~5.2]{FuH}.   As before,  $\mathsf{c}_\mu$, $\mu \in \Lambda(\GG)$,  is a representative of
the conjugacy class $\mathcal{C}_\mu$ of $\GG$.
\begin{equation}\label{tab:sl}
\begin{tabular}[t]{|c||c|c|c|c|c|}
\hline
        $\mathsf{c}_\mu$
       & $\pm \mathrm{I}$ & $\mathsf{u}_x, \ x \neq \pm 1$ & $ \mathsf{v}_y, \ y = 1,\ve$ &
       $- \mathsf{v}_y, \ y=-1,-\ve$ & $\mathsf{w}_{x,y}, \  x \neq \pm 1$	\\
	\hline \hline
	no. of such classes
       & $2$ & $\half (q-3)$ & $2$ &
       $2$ & $\half (q-1)$ 	\\
	\hline
      $|\mathcal C_\mu|$ &
      1&\ $q(q+1)$ & $\half(q^2-1)$ & $\half(q^2-1)$ & $q(q-1)$ \\
      \hline
       $\chi_{{}_\VV}(\mathsf{c}_\mu)$
       & $q+1$ & \ $2$ & \  $1$ &
      \  $1$ & \ $0$  	\\
      \hline
       $\chi_{{}_{\VV_q}}(\mathsf{c}_\mu)$
       & $q$ & \ $1$ & \  $0$ &
      \  $0$ &$-1$  	\\

	\hline  \end{tabular}
\end{equation}
The order of $\GG =\SL_2(\FF_q)$ is $q(q-1)(q+1)$ and the order of its  Borel subgroup $\BB$  of upper triangular matrices is $q(q-1)$.  Therefore,
the induced $\GG$-module  $\VV =\mathsf{Ind}_{\BB}^{\GG} \BB_\zr $ has dimension $q+1$, and $\VV = \GG_\zr \oplus \VV_q$, where $\VV_q$ is the $q$-dimensional
irreducible Steinberg module for $\GG$.
Using Table \ref{tab:sl} and Theorem \ref{T:walk},  we have the next result.   \bigskip

\begin{thm}\label{T:SL2q} Assume $\GG = \SL_2(\FF_q)$, where $q$ is odd. \begin{itemize}
\item[{\rm (a)}]  For $\VV = \mathsf{Ind}_\BB^\GG \BB_\zr = \GG_\zr \oplus \VV_q$,  the
$\GG$-module over $\CC$  induced from the trivial module $\BB_\zr$  for the Borel subgroup $\BB$ of upper-triangular matrices in $\GG$,  we have
\begin{equation}\dimm (\VV^{\ot k})^\GG
 =\begin{cases} \quad \  \ \, 1  &   \quad \hbox{\rm when\ $k=0$}\\
\displaystyle{\frac{1}{q(q-1)}\left( 2(q+1)^{k-1} +q(q-3)\cdot 2^{k-1} + 2(q-1)\right)} &  \quad \hbox{\rm when\ $k \geq 1$}.\end{cases} \end{equation}  The Poincar\'e series for the $\GG$-invariants $\mathsf{T}(\VV)^\GG$  in $\mathsf{T}(\VV) = \bigoplus_{k=0}^\infty \VV^{\ot k}$ is
\begin{equation}\mathsf{P}^0(t) = \sum_{k=0}^\infty   \dimm (\VV^{\ot k})^\GG \,  t^k
= \frac{1 - (q+3)t + (2q+3) t^2 - (q-1) t^3}{(1-t)\,(1-2t)\,(1-(q+1)t)}.  \\
 \end{equation}
 \item[{\rm (b)}]  For the Steinberg module $\VV_q$,  $\dimm (\VV_q^{\ot k})^\GG = 1$ when $k=0$, and
 \begin{align}\begin{split}\label{eq:Stsl} \dimm (\VV_q^{\ot k})^\GG & = \frac{1}{2(q^2-1)}\left( 4q^{k-1} +(q-1)^2 (-1)^k  +  (q-3)(q+1)\right)  \quad \hbox{\rm when\ $k \geq 1$,} \\
& =\begin{cases}  \displaystyle{\frac{2(q^{2\ell}-1)}{q^2-1}  =  2\sum_{j=0}^{\ell-1}  q^{2j}}  & \qquad \hbox{\rm if \, \ $k = 2\ell + 1 \geq 1$,}\\
  \displaystyle{1+2q \frac{(q^{2\ell-2}-1)}{q^2-1} = 1 + 2\sum_{j=0}^{\ell-2}  q^{2j+1}}  & \qquad \hbox{\rm if \, \ $k = 2\ell \geq 2$}.
 \end{cases} \end{split}
\end{align}
\item[{\rm (b)}] The Poincar\'e series $\mathsf{P}_q^0(t)$  for the $\GG$-invariants $\mathsf{T}(\VV_q)^\GG$  in $\mathsf{T}(\VV_q) = \bigoplus_{k=0}^\infty \VV_q^{\ot k}$ is
\begin{equation}\mathsf{P}_q^0(t) = \sum_{k=0}^\infty   \dimm (\VV^{\ot k})^\GG \,  t^k
= \frac{1 - qt +2 t^3}{(1+t)\,(1-t)\,(1-qt)}.
 \end{equation}
 \end{itemize}
\end{thm}

\begin{proof}  The proofs are analogous to those for Theorem \ref{T:GL2q} and are left to the reader.  \end{proof}

\section {The case $\GG$ is abelian and exponential generating functions}\label{S:abelian}
It is convenient to regard an arbitrary finite abelian group $(\GG,+)$ as
a multiplicative group and write  $\es^\arm$ for $\arm \in \GG$,  so that the group operation is
given by $\es^\arm \es^{\br}  = \es^{\arm+\br}$, $\arm, \br \in \GG$, where the
sum $\arm+\br$ is addition in $\GG$.    The identity element is  $\es^{\zr}$.
Since $\GG$ is abelian,   the irreducible $\GG$-modules are all one-dimensional,
and we label them and the conjugacy classes with the elements of $\GG$.     Thus, for $\arm  \in \GG$,   let  $\GG_\arm= \CC x_\arm$,  where
$ \es^\br x_\arm  = \chi_{\arm}(\br) x_\arm$, and let $\chi_{\arm}$ denote the corresponding character.
The characters satisfy
\begin{align}\chi_{\arm}(\br + \br') &= \chi_{\arm}(\br) \chi_{\arm}(\br') \quad \ \  \hbox{\rm for all \ $\arm, \br, \br' \in \GG$, and}\\
\chi_{\arm+\arm'}(\br) &= \chi_{\arm}(\br) \chi_{\arm'}(\br) \qquad \hbox{\rm for all \  $\arm,\arm',\br\in \GG$,} \label{eq:tens}\end{align}
as $\GG_\arm \ot \GG_{\arm'} \cong  \GG_{\arm+\arm'}$.
for all $\arm, \arm' \in \GG$.
Since $\chi_\arm(\br)\chi_{-\arm}(\br) = \chi_{\arm - \arm'}(\br) = \chi_{\zr}(\br) = 1$   and
$\chi_{\arm}(\zr) = 1$ for all $\arm,\br \in \GG$,  the following hold:
\begin{align}\begin{split} \label{eq:charrels} \chi_{-\arm}(\br) &= \chi_{\arm}(\br)^{-1} = \overbar{\chi_{\arm}(\br)}  \\
\chi_{\arm}(-\br) &= \chi_{\arm}(\br)^{-1} = \overbar{\chi_{\arm}(\br)}. \end{split}
\end{align}

By the fundamental theorem of finite abelian groups,  we may suppose that
$\GG = \ZZ_{r_1} \times  \ZZ_{r_2} \times \cdots \times \ZZ_{r_n}$  where the $r_j$ are powers of
not necessarily distinct primes.     The elements of $\GG$  have the form $\es^{\br}$, where
$\br = (b_1,b_2, \dots, b_n)$ and $b_j \in \ZZ_{r_j}$ for each $j$.  Set $\om_j = \er^{2\pi i/r_j}$.
Then $\GG_{\arm} = \CC x_\arm$,  where
\begin{equation}\es^\br x_{\arm} = \chi_\arm(\br) x_{\arm} \quad \hbox{\rm and \quad  $\chi_{\arm}(\br) = \om_1^{a_1b_1}\om_2^{a_2b_2} \cdots \om_n^{a_n b_n}$}.\end{equation}

 Let $\ve_j$ be the $n$-tuple
with $1$ in position $j$ and $0$ for all its other components.    Here we suppose that $\VV = \GG_{\ve_1} \oplus
\GG_{\ve_2} \oplus  \cdots \oplus \GG_{\ve_n}$, so for $\br = (b_1,b_2,\ldots, b_n) \in \GG$, the character values are  given by
\begin{equation} \label{eq:chiVV}  \chi_{{}_{\VV}}(\br)  = \sum_{j=1}^n  \chi_{\ve_j}(\br) = \sum_{j=1}^n \om_j^{b_j}\qquad \quad
 \chi_{{}_{\VV^{\ot k}}}(\br) = \chi_{{}_{\VV}}(\br)^k = \left(\sum_{j=1}^n\om_j^{b_j}\right)^k.  \end{equation}

We have the following corollary to Theorem \ref{T:walk}:
\medskip

\begin{cor}\label{C:graph}  The number of walks of $k$-steps from node $\arm$ to node  $\cm$ on the McKay quiver  $\cR_{\VV}(\GG)$  for  $\GG = \ZZ_{r_1} \times  \ZZ_{r_2} \times \cdots \times \ZZ_{r_n}$  and
$\VV = \GG_{\ve_1} \oplus \GG_{\ve_2} \oplus \cdots \oplus \GG_{\ve_n}$   is
\begin{align}\label{eq:Zrwalkcount}
\begin{split} (\A^k)_{\arm, \cm}  &=  \sum_{0 \leq \ell_1,\ell_2,\dots, \ell_n \leq k}
\binom{k}{{\ell_1, \ell_2, \dots, \ell_n}}
\end{split}
\end{align}
where the sum is over all  $\ell_1,\ell_2, \dots, \ell_n$ such that $\ell_1 +\ell_2+ \cdots + \ell_n = k$ and
$c_i - a_i \equiv \ell_i   \modd r_i$ for all $i\in [1,n] =\{1,2,\dots, n\}$.
\end{cor}

\begin{proof}  Now
\begin{align}\label{eq:Zrwalkcount2} \begin{split} (\A^k)_{\arm, \cm}  &=  \sum_{0 \leq \ell_1,\dots, \ell_n \leq k}
\vert \GG \vert^{-1} \sum_{\br \in\GG}\chi_{\arm}(\br) \chi_{{}_{\VV}}^k(\br) \overbar{\chi_{\cm}(\br)}
=\vert \GG \vert^{-1} \sum_{\br \in\GG}\chi_{\arm-\cm}(\br) \chi_{{}_{\VV}}^k(\br) \\
& = \vert \GG \vert^{-1}\sum_{\br \in\GG}\om_1^{(a_1-c_1)b_1} \cdots \, \om_n^{(a_n-c_n)b_n} \left(\sum_{j=1}^n \om_j^{b_j}\right)^k \\
& = \vert \GG \vert^{-1}\sum_{\br \in\GG}\om_1^{(a_1-c_1)b_1} \cdots \, \om_n^{(a_n-c_n)b_n}\left( \sum_{0 \leq \ell_1,\dots, \ell_n \leq k}
\binom{k}{{\ell_1,\ell_2, \dots, \ell_n}}   \om_1^{\ell_1b_1} \cdots \om_n^{\ell_n b_n}\right) \\
& =  \vert \GG \vert^{-1}\sum_{0 \leq \ell_1,\ell_2, \ldots, \ell_n\leq k} \binom{k}{{\ell_1,\ell_2, \dots, \ell_n}}\left( \left( \sum_{b_1 = 0}^{r_1-1}  \om_1^{(a_1-c_1+\ell_1)b_1}\right) \  \cdots \  \left(\sum_{b_n = 0}^{r_n-1}  \om_n^{(a_n-c_n+\ell_n)b_n}\right)\right)    \\
& = \sum_{0 \leq \ell_1,\ell_2, \ldots, \ell_n\leq k} \binom{k}{{\ell_1,\ell_2, \dots, \ell_n}}  \end{split} \end{align}
by applying \eqref{eq:omsum} repeatedly,  where the sum is over all  $\ell_1,\ell_2, \dots, \ell_n$ such that $\ell_1 +\ell_2+ \cdots + \ell_n = k$ and
$\ell_i \equiv c_i - a_i \modd r_i$ for all $i \in [1,n]$.      \end{proof}

 \begin{subsection}{Exponential generating functions} \end{subsection}

For $\cm \in \GG = \ZZ_{r_1} \times  \ZZ_{r_2} \times \cdots \times \ZZ_{r_n}$
and $\VV = \GG_{\ve_1} \oplus \GG_{\ve_2} \oplus \cdots \oplus \GG_{\ve_n}$, let

$$\gr^{\cm}(t) : =  \sum_{k=0}^{\infty} (\A^k)_{\zr,\cm} \frac{t^k}{k!}$$
denote the exponential generating function for walks of $k$ steps from $\zr$ to $\cm$ on
the McKay quiver $\cR_{\VV}(\GG)$ (and also for the multiplicity of
$\GG_\cm$ in $\VV^{\ot k}$ and for
dimension of the irreducible module $\Zs_k^\cm(\GG)$ for the centralizer algebra).   We determine an expression for
$\gr^{\cm}(t)$ in terms of generalized hyperbolic functions.

The \emph{generalized hyperbolic function}   $\hr_j(t,r)$ for $j\in \ZZ$  is defined by
\begin{equation} \hr_j(t,r)  :=   r^{-1} \sum_{m=0}^{r-1}  \om^{(1-j)m} \er^{\om^m t},\end{equation}
where $\om = \er^{2\pi i/r}$.  In particular,
\begin{equation}\label{eq:h1} \hr_1(t,r) = r^{-1} \sum_{m=0}^{r-1} \er^{\om^m t},\end{equation}
so that  $\hr_1(t,1) = \er^t$ and $\hr_1(t,2) = \cosh\,t.$    Because
$$\hr_{j+r}(t,r) =  \hr_{j}(t,r) \qquad \hbox{\rm for} \  j \in \ZZ,$$
there are $r$ distinct generalized hyperbolic functions $\hr_{j}(t,r)$ for a fixed value of $r$.
\medskip

\begin{thm}\label{T:genfnc}   For $\GG = \ZZ_{r_1} \times \ZZ_{r_2} \times \cdots \times \ZZ_{r_n}$ and $\cm = (c_1,c_2,\dots, c_n) \in \GG$,  the exponential generating function for the number of walks of $k$ steps from $\zr$ to $\cm$ on $\cR_{\VV}(\GG)$ is
$$\gr^{\cm}(t)  =  \sum_{k=0}^{\infty} (\A^k)_{\zr, \cm}  \frac{t^k}{k!} =  \hr_{1+c_1}(t, r_1) \hr_{1+c_2}(t,r_2)\, \cdots \, \hr_{1+c_n}(t,r_n).$$
\end{thm}

Before giving the proof,  we note the following immediate consequences.

\begin{cor}  For $\GG = \ZZ_{r_1} \times  \ZZ_{r_2} \times \cdots \times \ZZ_{r_n}$
and $\VV = \GG_{\ve_1} \oplus \GG_{\ve_2} \oplus \cdots \oplus \GG_{\ve_n}$,
   \begin{itemize} \item[{\rm (a)}] $\displaystyle{\gr^{\zr}(t)  =  \sum_{k=0}^{\infty} (\A^k)_{\zr,\zr} \frac{t^k}{k!} = \hr_{1}(t, r_1) \hr_{1}(t,r_2)\, \cdots \, \hr_{1}(t,r_n).}$
\item[{\rm (b)}] When $\GG = \ZZ_r^n$, then $\displaystyle{\gr^{\zr}(t)  = \hr_1(t,r)^n.}$
\end{itemize}  \end{cor}
\begin{remark} {\rm Part (b) of this  corollary  generalizes \cite[Cor.~4.29]{BM}, which
says that the generating function for the number of walks on a hypercube of order $n$  is given by  $\gr^{\zr}(t) = (\cosh\,t)^n = \hr_1(t,2)^n$.
Theorem 4.25 of \cite{BM} shows that for $\ZZ_2^n$,
$$\gr^{\cm}(t) =  (\cosh\,t)^{n-\mathfrak{h}(\cm)} (\sinh\, t)^{\mathfrak{h}(\cm)},$$
where $\mathfrak h(\cm)$ is the Hamming weight of $\cm$ (the number of ones  in $\cm$).
This follows directly from Theorem \ref{T:genfnc},  since each component  of  $\cm$ equal to 1 contributes  a factor  $ \hr_2(t,2) = \sinh\, t$,
and each component of $\cm$ equal to 0 gives a factor $\hr_1(t,2) = \cosh\, t$.}
 \end{remark}

\noindent \emph {Proof of Theorem \ref{T:genfnc}.} \ \  Observe  that by \eqref{eq:chiVV} and Corollary \ref{C:graph},
\begin{align*} \gr^{\cm}(t)  &= \sum_{k=0}^{\infty} (\A)^k_{\zr, \cm} \, \frac{t^k}{k!}   \\
&= \vert \GG \vert^{-1}  \sum_{k=0}^\infty \, \sum_{\br = (b_1,\dots, b_n)  \in \GG}\om_1^{-b_1c_1}\, \cdots\, \om_n^{-b_nc_n} \left(\sum_{j=1}^n \om_j^{b_j}\right)^k \frac{t^k}{k!}  \\
&= r_1^{-1} \cdots r_n^{-1} \sum_{k=0}^\infty \, \sum_{\br \in \GG} \om_1^{-b_1c_1} \cdots \om_n^{-b_nc_n}  \left( \sum_{\ell_1+\cdots +\ell_n= k}
\frac{\om_1^{b_1\ell_1 } t^{\ell_1}}{\ell_1 \,!} \ \cdots \ \frac{\om_n^{b_n\ell_n} t^{\ell_n}}{\ell_n \,!} \right)   \\
&= \left(r_1^{-1} \sum_{b_1=0}^{r_1-1} \sum_{\ell_1=0}^\infty \om^{-b_1c_1}  \frac{\om_1^{b_1\ell_1 } t^{\ell_1}}{\ell_1 \,!} \right)  \times
\cdots  \times  \left(r_n^{-1}\sum_{b_n=0}^{r_n-1} \sum_{\ell_n=0}^\infty  \om_n^{-b_nc_n}   \frac{\om_n^{b_n\ell_n } t^{\ell_n}}{\ell_n \,!} \right) \\
&=  \left(r_1^{-1}\sum_{b_1=0}^{r_1-1}  \om_1^{-b_1c_1}  \er^{\om_1^{b_1} t} \right)  \ \times
\cdots  \ \times \  \left(r_n^{-1}\sum_{b_n=0}^{r_n-1}  \om^{-b_nc_n}  \er^{\om_n^{b_n} t} \right) \\
& =  \hr_{1+c_1}(t, r_1)\, \hr_{1+c_2}(t,r_2)\ \cdots \ \hr_{1+c_n}(t,r_n).  \hspace{4cm} \square \end{align*}
\medskip

Using  \eqref{eq:omsum}
and the definition of the generalized hyperbolic function  $\hr_j(t,r)$,
one sees that the Taylor series expansion of  $\hr_j(t,r)$ is given by

\begin{equation}\label{eq:powerser} \hr_j(t,r) = \sum_{q=0}^\infty  \frac {t^{qr+j-1}}{(qr + j-1)!} \end{equation}

Suppose  $\cm = (c_1,c_2, \dots, c_n) \in \GG= \ZZ_{r_1} \times \ZZ_{r_2} \times \cdots \times \ZZ_{r_n}$, where $0 \leq c_j < r_j$ for all $j$,
 and let $\vert \cm \vert =  \sum_{j=1}^n c_j$.    We have shown in Theorem \ref{T:genfnc}  that the exponential generating function $\gr^{\cm}(t)$   is given by
 $$\gr^{\cm}(t)  =  \sum_{k=0}^{\infty} (\A^k)_{\zr,\cm}\ \frac{t^k}{k!} =  \hr_{1+c_1}(t, r_1) \hr_{1+c_2}(t,r_2)\, \cdots \, \hr_{1+c_n}(t,r_n).$$
Combining that with the expressions coming from \eqref{eq:powerser}, we have
\begin{align*}  \gr^{\cm}(t)  & = \hr_{1+c_1}(t, r_1) \hr_{1+c_2}(t,r_2)\, \cdots \, \hr_{1+c_n}(t,r_n) \\
&= \left(\sum_{q_1=0}^\infty  \frac{t^{q_1 r_1 + c_1}}{(q_1r_1+c_1)!}\right)
\left(\sum_{q_2=0}^\infty  \frac{t^{q_2 r_2 + c_2}}{(q_2r_2+c_2)!}\right) \cdots
\left(\sum_{\ell_n=0}^\infty  \frac{t^{q_n r_n+ c_n}}{(q_n r_n+c_n)!}\right)\\
&= \sum_{k=0}^\infty \ \sum_{q_1r_1+ \cdots + q_n r_n+|\cm|=k}  \frac{k!}{(q_1r_1 + c_1)!(q_2r_2+c_2)! \cdots (q_n r_n+c_n)!}\,\,
\frac{t^k}{k!}    \end{align*}
 Setting $q_i r_i + c_i = \ell_i$ for $i=1,2,\dots,n$ gives the result in Corollary \ref{C:graph} with $\arm = \zr$,  which provides
a formula for the dimension of the irreducible module $\Zs_k^\cm(\GG)$ for the centralizer algebra $\Zs_k(\GG)$:
   \begin{equation}
 \label{eq:C:dimswalks} \dimm \Zs_k^{\cm}(\GG) = (\A^k)_{\zr,\cm}
= \sum_{0 \leq \ell_1,\ell_2, \dots, \ell_n \leq k}   \binom{k}{{\ell_1, \ell_2, \ldots, \ell_n}}.
  \end{equation}
The sum is over all $0 \leq \ell_1,\ell_2, \dots, \ell_n\leq k$ such that $\ell_1+\ell_2+ \cdots + \ell_n = k$ and
$\ell_i \equiv c_i \mod r_i$ for all $i\in [1,n]$.
In particular,  when  $\GG= \ZZ_{r_1} \times \ZZ_{r_2} \times \cdots \times \ZZ_{r_n}$  and $\cm = \zr$,  then
\begin{equation}\dimm (\VV^{\ot k})^\GG =  \dimm \Zs_k^{\zr}(\GG) = \sum_{0 \leq \ell_1,\ell_2, \dots,  \ell_n\leq k} \binom{k}{{\ell_1, \ \ell_2, \  \ldots, \  \ell_n}},\end{equation}
where $\ell_1 +\ell_2+ \cdots+\ell_n = k$ and $\ell_i \equiv 0 \modd r_i$   for all $i \in [1,n]$.

An alternate approach to the result in \eqref{eq:C:dimswalks}  is
via characters.
For $\GG = \ZZ_{r_1} \times \cdots \times \ZZ_{r_n}$  and
$\VV = \GG_{\ve_1}\oplus \GG_{\ve_2} \oplus \cdots \oplus\GG_{\ve_n}$, where $\GG_{\ve_j} = \CC x_{\ve_j}$ for all $j$, the character of the $k$th tensor power of $\VV$ is given by

\begin{align*}  \chi_{\VV^{\ot k}} &=   \chi_{{}_{\VV}}^k =  (\chi_{\ve_1} +\cdots + \chi_{\ve_n})^k \\
&= \sum_{\substack{0 \leq \ell_1,\ell_2,\dots, \ell_n \leq k \\ \ell_1 +\ell_2+ \cdots + \ell_n = k}}  \binom{k}{ {\ell_1,\ell_2,\, \dots,\, \ell_n}}  \chi_{\ve_1}^{\ell_1} \cdots   \chi_{\ve_n}^{\ell_n} \\
&=  \sum_{\substack{0 \leq \ell_1,\ell_2,\dots, \ell_n \leq k \\ \ell_1 +\ell_2+ \cdots + \ell_n=k}} \binom{k}{{\ell_1,\ell_2, \, \dots,\,\ell_n}}   \chi_{\ell_1\ve_1+\ell_2\ve_2+ \cdots + \ell_n\ve_n}.
\end{align*}
Now for $\cm = (c_1,c_2, \dots, c_n)$ with $0 \leq c_i < r_i$ for all $i \in [1,n]$, the multiplicity of the
character $\chi_\cm$ in this expression is exactly the number of $n$-tuples $(\ell_1, \ell_2, \ldots, \ell_n)$ such that
$\ell_i \equiv c_i \modd r_i$  for all $i \in [1,n]$,  as in \eqref{eq:C:dimswalks}.

\begin{example}\label{ex:Z4Z2}  Consider $\GG = \ZZ_4 \times \ZZ_2$ and the tensor power $\VV^{\ot 6}$ for
$\VV = \GG_{\ve_1} \oplus \GG_{\ve_2}$.   Then
\begin{align*} (\chi_{\ve_1} + \chi_{\ve_2})^6 &= \chi_{6\ve_1} + 6 \chi_{5 \ve_1+\ve_2} + 15 \chi_{4\ve_1+2\ve_2}
+ 20 \chi_{3\ve_1 +3\ve_2} \\ & \qquad \quad
+ 15 \chi_{2\ve_1 + 4\ve_2}  + 6 \chi_{\ve_1 + 5\ve_2} + \chi_{6\ve_2} \\
&=16 \chi_{2 \ve_1} + 12 \chi_{\ve_1+\ve_2} + 16 \chi_{\zr} + 20 \chi_{3\ve_1+\ve_2}. \end{align*}
Thus, $\dimm \Zs_6^{(2,0)}(\GG) = 16$, \  $\dimm \Zs_6^{(1,1)}(\GG) = 12$, \  $\dimm \Zs_6^{(0,0)}(\GG) = 16$,
and $\dimm \Zs_6^{(3,1)}(\GG) = 20$.
\end{example}

\medskip
\begin{subsection}{The Bratteli diagram and a basis for $\Zs_k(\GG)$ when $\GG = \ZZ_{r_1} \times  \ZZ_{r_2} \times \cdots \times \ZZ_{r_n}$  and \newline  $\VV = \GG_{\ve_1} \oplus \GG_{\ve_2} \oplus \cdots \oplus \GG_{\ve_n}$}\label{S:abelbasis} \end{subsection}

A walk of $k$ steps on the McKay quiver  $\cR_{\VV}(\GG)$ from $\zr$ to $\cm$ corresponds to a path
$\left(\cm^{(0)},{\cm^{(1)}},  \ldots, {\cm^{(k)}}\right)$ on the Bratteli diagram $\mathcal{B}_{\VV}(\GG)$
 starting at $\cm^{(0)} = \zr= (0,0,\ldots,0)$ at level 0 and ending at $\cm =\cm^{(k)}$ at level $k$
 such that  $\cm^{(i)} \in \GG$   for each $1 \le i \le k$, and  $\cm^{(i)}= \cm^{(i-1)} + \ve_{\gamma_i}$ for some $\gamma_i \in [1,n]$, where
$\cm^{(i)}$ is connected to $ \cm^{(i-1)}$   by the edge corresponding to $\gamma_i$  in $\cR_{\VV}(\GG)$.
The subscript on node  $\cm$ at level $k$  in $\mathcal{B}_\VV(\GG)$  indicates  the number
 of such paths, which is  the multiplicity of
the irreducible $\GG$-module $\GG_\cm$  in $\VV^{\ot k}$ and also equal to the
dimension of the irreducible $\Zs_k(\GG)$-module  $\Zs_k^\cm(\GG)$.  The sum of the squares of
those dimensions at level $k$ is the number on the right, which is the dimension of the centralizer algebra $ \Zs_k(\GG)$.   Levels 0,1,\dots,6  of the Bratteli diagram for $\ZZ_4 \times \ZZ_2$ are displayed in Figure \ref{fig:bratteli2}.
The nodes of the diagram  correspond to elements $\cm = (c_1,c_2) \in \ZZ_4 \times \ZZ_2$
and have $c_1 \in \{0,1,2,3\}$ and $c_2 \in \{0,1\}$.

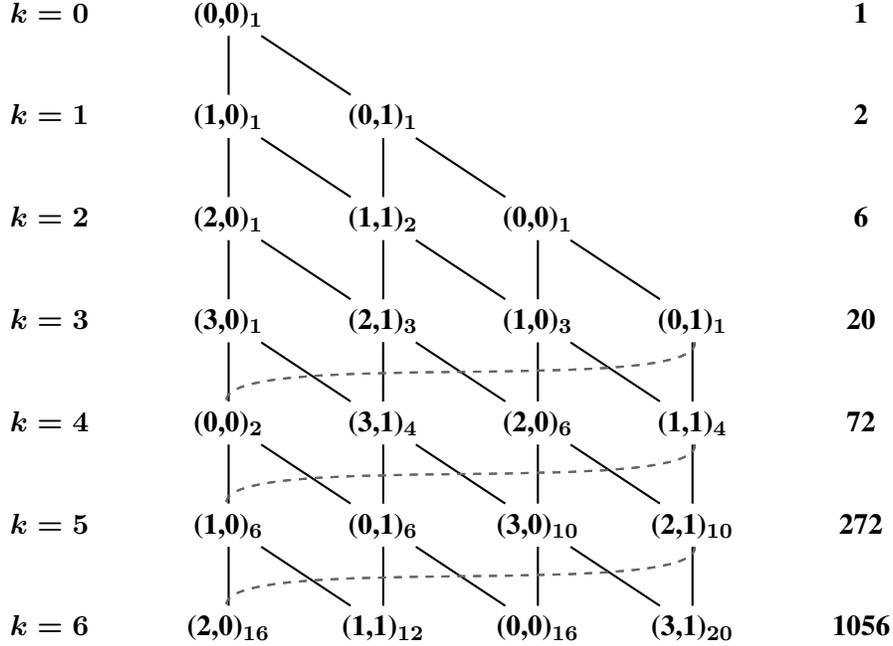
\begin{figure}
$$\begin{tikzpicture}[>=latex,text height=1.5ex,text depth=0.25ex]
    % "text height" and "text depth" are required to vertically
    % align the labels with and without indices.

  % The various elements are conveniently placed using a matrix:
  \matrix[row sep=0.8cm,column sep=0.35cm] {
% Oth line: System noise & input matrix
        \node (k0) {$\bs{k=0}$}; &&&
        \node (V00) {{\bf (0,0)}${}_{\color{black} \bf 1}$};  &&&&&&&&&
        \node(d0) {\bf \color{black}1\color{black}};
        \\
% First line: System noise & input matrix
        \node (k1)  {$\bs{k=1}$}; &&&
        \node (V14) {{\bf (1,0)}${}_{\color{black} \bf 1}$};    &&
        \node (V11) {{\bf (0,1)}${}_{\color{black} \bf 1}$};  &  &&&&&&
        \node(d1) {\bf \color{black}2\color{black}};
        \\
% Second line: System noise & input matrix
        \node (k2)  {$\bs{k=2}$}; &&&
        \node (V23){{\bf (2,0)}${}_{\color{black} \bf 1}$};   &&
        \node (V20){{\bf (1,1)}${}_{\color{black} \bf 2}$};  &&
        \node (V22){{\bf (0,0)}${}_{\color{black} \bf 1}$};    &&&&&
         \node(d2) {\bf \color{black}6\color{black}};
        \\
% Third line: System noise & input matrix
        \node (k3)  {$\bs{k=3}$}; &&&
        \node (V32){{\bf (3,0)}${}_{\color{black} \bf 1}$};   &&
        \node (V34) {{\bf (2,1)}${}_{\color{black} \bf 3}$};    &&
        \node (V31) {{\bf (1,0)}${}_{\color{black} \bf 3}$};    &&
        \node (V33) {{\bf (0,1)}${}_{\color{black} \bf 1}$};    &&&
       \node(d3) {\bf \color{black}20\color{black}};
   \\
% Fourth line: System noise & input matrix
        \node (k4)  {$\bs{k=4}$}; &&&
       % \node (V41) {\bf 1};  &&
        \node (V43) {{\bf (0,0)}${}_{\color{black} \bf 2}$};  &&
        \node (V44){{\bf (3,1)}${}_{\color{black} \bf 4}$};  &&
        \node (V42) {{\bf (2,0)}${}_{\color{black} \bf 6}$};   &&
        \node (V40){{\bf (1,1)}${}_{\color{black} \bf 4}$};   &&&
        \node(d4) {\bf \color{black}72\color{black}};
        \\
        % Fifth line: System noise & input matrix
        \node (k5)  {$\bs{k=5}$}; &&&
        \node (V50){{\bf (1,0)}${}_{\color{black} \bf 6}$};   &&
        \node (V52) {{\bf (0,1)}${}_{\color{black} \bf 6}$};  &&
        \node (V54) {{\bf (3,0)}${}_{\color{black} \bf 10}$};    &&
        \node (V51) {{\bf (2,1)}${}_{\color{black} \bf 10}$};    &&&
     \node(d5) {\bf \color{black}272\color{black}};
        \\
   % Sixth line: System noise & input matrix
        \node (k6)  {$\bs{k=6}$}; & &&
        \node (V63) {{\bf (2,0)}${}_{\color{black} \mathbf{16}}$};  & &
        \node (V60) {{\bf (1,1)}${}_{\color{black} \mathbf{12}}$};   &&
        \node (V62) {{\bf (0,0)}${}_{\color{black} \mathbf{16}}$};  &&
        \node (V64) {{\bf (3,1)}${}_{\color{black} \mathbf{20}}$};  &&&
        \node(d6) {\bf \color{black}1056\color{black}};
        \\
        };
  \path

        (V00) edge[thick] (V14)
        (V00) edge[thick] (V11)

        (V14) edge[thick] (V23)
        (V14) edge[thick] (V20)
        (V11) edge[thick] (V20)
        (V11) edge[thick] (V22)

        (V32) edge[thick] (V23)
        (V34) edge[thick] (V23)
        (V34) edge[thick] (V20)
        (V31) edge[thick] (V20)
        (V31) edge[thick] (V22)
        (V33) edge[thick] (V22)	

        (V32) edge[thick] (V43)
        (V32) edge[thick] (V44)
        (V34) edge[thick] (V42)
        (V34) edge[thick] (V44)
        (V31) edge[thick] (V42)
        (V31) edge[thick] (V40)
        (V34) edge[thick] (V44)
        (V33) edge[thick] (V40)

        (V43) edge[thick] (V50)
        (V43) edge[thick] (V52)
        (V44) edge[thick] (V52)
        (V44) edge[thick] (V54)
        (V42) edge[thick] (V54)
        (V42) edge[thick] (V51)
        (V40) edge[thick] (V51)

        (V50) edge[thick] (V63)
        (V50) edge[thick] (V60)
        (V52) edge[thick] (V60)
        (V52) edge[thick] (V62)
        (V54) edge[thick] (V62)
        (V54) edge[thick] (V64)
        (V51) edge[thick] (V64);

        \draw[black!60,thick,dashed]  (V33)  .. controls +(.1,-1) and +(-.1,1) .. (V43);
        \draw[black!60,thick,dashed]  (V33)  .. controls +(.1,-1) and +(-.1,1) .. (V43);
        \draw[black!60,thick,dashed]  (V40)  .. controls +(.1,-1) and +(-.1,1) .. (V50);
        \draw[black!60,thick,dashed]  (V40)  .. controls +(.1,-1) and +(-.1,1) .. (V50);
        \draw[black!60,thick,dashed]  (V51)  .. controls +(.1,-1) and +(-.1,1) .. (V63);
        \draw[black!60,thick,dashed]  (V51)  .. controls +(.1,-1) and +(-.1,1) .. (V63);
        \\
        ;
\end{tikzpicture}$$
\caption{Levels $k=0,1,\dots,6$  of the Bratteli diagram for $\ZZ_4 \times \ZZ_2$} \label{fig:bratteli2}
\end{figure}
\medskip

\begin{remark}  {\rm The subscripts in
 the last row of the Bratteli diagram in Figure \ref{fig:bratteli2},
exactly match with the dimensions determined in Example \ref{ex:Z4Z2}.
The sequence of numbers in the right-hand column of  Figure \ref{fig:bratteli2}
(i.e. the dimension $d(k)$ of the centralizer algebra $\Zs_k(\ZZ_4 \times \ZZ_2)$)
satisfies $d(k) = a(k-1)$ in sequence \cite[A063376]{OEIS},
where  $a(-1) = 1$ and $a(k-1) =  2^{k-1} + 4^{k-1}$ for $k \geq 1$.   Among the objects  that $a(k-1)$ counts
is  the number of  closed walks of length $2k$ at a vertex of the circular graph on 8 nodes, which is
the same as $\dimm \Zs_k(\GG)$ for $\GG = \ZZ_8$ and $\VV = \GG_1 \oplus \GG_7$ (see Section \ref{SS:cyclic}).}
 \end{remark}

Much of the next result is evident from the above considerations.   \medskip

\begin{thm}\label{T:basis}   Assume $\GG = \ZZ_{r_1} \times \ZZ_{r_2} \times \cdots \times \ZZ_{r_n}$ and
$\VV = \GG_{\ve_1} \oplus \GG_{\ve_2} \oplus \cdots \oplus \GG_{\ve_n}$.   Then the following hold:
\begin{itemize} \item[{\rm (i)}]  For $\cm = (c_1, \dots, c_n) \in \GG$,  a basis for the irreducible $\Zs_k(\GG)$-module
$\Zs_k^\cm(\GG) \subseteq \VV^{\ot k}$ is
$$
\left \{x(\gamma) := x_{\ve_{\gamma_1}} \ot \cdots \ot x_{\ve_{\gamma_k}} \, \big | \,
\gamma_i   \in  [1,n] \,  \text{ for all  $i \in \, [1,k]$, and  $\textstyle{ \sum_{i=1}^k} \ve_{\gamma_i} = \cm$} \right \}.$$
 \item[{\rm (ii)}]  $\es^\arm x(\gamma) =  \chi_\cm(\arm) x(\gamma)$ for all $\arm\in \GG$ and all $x(\gamma)$ in {\rm (i)},
 where $\chi_{\cm}(\arm) = \prod_{j=1}^n  \om_j^{a_jc_j}$ and $\om_j = \er^{2\pi i/r_j}$ for all $j \in [1,n]$,
 so that $\Zs_k^\cm(\GG)$ is  also a $\GG$-submodule of $\VV^{\ot k}$;  it is the sum of all the copies of  the  irreducible
 $\GG$-module $\GG_\cm$ in $\VV^{\ot k}$.
\item[{\rm (iii)}] For  $\gamma =(\gamma_1,\dots, \gamma_k),\beta = (\beta_1,\dots,\beta_k)  \in [1,n]^k$ with
$\sum_{i=1}^k \ve_{\gamma_i} = \sum_{i=1}^k \ve_{\beta_i}$,  let $\EE_{\gamma}^\beta \in \End(\VV^{\ot k})$ be defined by  $\EE_{\gamma}^\beta x(\alpha)= \delta_{\alpha,\gamma} x(\beta)$ for $\alpha \in [1,n]^k$.   Then
$ \EE_{\vartheta}^\eta  \EE_{\gamma}^\beta= \delta_{\beta,\vartheta} \EE_{\eta}^\beta$
for all such $\vartheta,\eta$, and the $\EE_\gamma^\beta$  determine a basis for
 $\Zs_k(\GG) = \End_{\GG}(\VV^{\ot k})$. \end{itemize} \end{thm}

\begin{proof}  From the calculation below it is easy to see that the transformations $\EE_\gamma^\beta$
for $\gamma,\beta \in [1,n]^k$  as in  {\rm (iii)} of Theorem \ref{T:basis}
commute with the action of $\GG$ on $\VV^{\ot k}$, hence belong to $\Zs_k(\GG)$.    Indeed, suppose $\alpha \in [1,n]^k$ with
 $\sum_{i=1}^k \ve_{\alpha_i} = \cm' \in \GG$, and assume  $\arm \in \GG$.
Then
\begin{align*}  \es^\arm \EE_{\gamma}^\beta \big(x(\alpha)\big) &=  \delta_{\alpha,\gamma} \es^{\arm} x(\beta) =
\delta_{\alpha,\gamma} \chi_{\cm}(\arm) x(\beta) \\
\EE_{\gamma}^\beta \es^\arm \big(x(\alpha)\big) & = \chi_{\cm'}(\arm)\delta_{\alpha,\gamma} x(\beta).  \end{align*}
Both expressions are 0 when  $\alpha \neq \gamma$, and when $\alpha = \gamma$,  then $\cm' = \cm$,
and the two expressions are identical.  The transformations $\EE_{\gamma}^\beta$ are clearly linearly independent.
The number of  $\gamma = (\gamma_1,\gamma_2, \dots, \gamma_k)$ such that $\sum_{i=1}^k \ve_{\gamma_i} = \cm$  is the number of paths from $\zr$ at level 0 to $\cm$ at level $k$ of the
Bratteli diagram $\mathcal{B}_\VV(\GG)$, which is $\dimm \Zs_k^\cm(\GG)$.   Therefore, the number of $\EE_{\gamma}^\beta$ in (iii) equals $\big(\dimm \Zs_k^{\cm}(\GG)\big)^2$, and since
$\dimm \Zs_k(\GG) = \sum_{\cm \in \GG} \big(\dimm \Zs_k^{\cm}(\GG)\big)^2$,
 taking the union of the sets of transformations $\EE_{\gamma}^\beta$
as $\cm$ ranges over all the elements of $\GG$  will give
 a basis for $\Zs_k(\GG)$.
\end{proof}

\begin{remark}{\rm The condition $\sum_{i=1}^k \ve_{\gamma_i} = \sum_{i=1}^k \ve_{\beta_i}$ in Theorem \ref{T:basis}  is equivalent to saying
$(\# \gamma_i = j) \equiv (\# \beta_i = j) \modd r_j$ for all $j=1,\dots, n$.   That interpretation leads to the diagrammatic
point of view that we describe next.} \end{remark}

%%%%%%%%%%%%%%%%%%%%%%%%%%%%%%%%%%%%%%%%%%%%%%%%%

  \begin{subsection}{A diagram basis for $\Zs_k(\GG)$ for $\GG = \ZZ_{r_1} \times \ZZ_{r_2} \times \cdots \times \ZZ_{r_n}$ } \label{S:abeldiag} \end{subsection}
 \medskip
 In this section, we present a realization $\Zs_k(\GG)$ as a diagram algebra.
 We identify the basis element  $\EE_{\gamma}^\beta$ with a
diagram having two rows of $k$ nodes.   The components of $\gamma = (\gamma_1,\dots, \gamma_k)$, which lie in $[1,n]$,  label the
nodes on the bottom row, and those of $\beta = (\beta_1, \dots, \beta_k)$
the top row.     Nodes having the same labels are connected, but
the way the edges are drawn is immaterial.  What matters is that nodes with identical labels are all connected somehow, and those with different labels  are not.
Thus, for $\gamma = (3,4,4,1,4,4,2,4,3,4,4,2)$ and $\beta = (2,4,1,3,1,2,2,4,1,2,2,3)$
in $[1,4]^{12}$,  the basis element $\EE_{\gamma}^\beta$ is identified with the diagram

\begin{equation}\label{eq:diag}
\vcenter{\hbox{%\vspace{1cm}
\begin{tikzpicture}[>=latex,text height=.5ex,text depth=0.1ex, scale=1.25]
\foreach \i in {0,...,12}
{ \path (\i,1) coordinate (T\i); \path (\i,0) coordinate (B\i); }
%%%%%%%
\filldraw[fill= gray!50,draw=gray!50,line width=8pt]  (T1) -- (T12) -- (B12) -- (B1) -- (T1);
%%%%%%%%%%%%%%%%%%%%%%%%%%%%
\draw[black] (T3) .. controls +(.1,-.25) and +(.1,-.25) .. (T5) ;
\draw[black] (T5) .. controls +(.1,-.5) and +(-.1,-.5) .. (T9);
\draw[black] (T3) .. controls +(0,-.25) and +(0,.25) .. (B4) ;
%%%%%%%%%%%%%%%%%%%%%%%%%%%%%
\draw[black] (T1) .. controls +(.1,-.48) and +(-.1,-.48) .. (T6);
\draw[black] (T6) .. controls +(.1,-.2) and +(-.1,-.2) .. (T7);
\draw[black] (T7) .. controls +(.1,-.35) and +(-.1,-.35) .. (T10);
\draw[black] (T10) .. controls +(.1,-.2) and +(-.1,-.2) .. (T11);
\draw[black] (T11) .. controls +(0,-.2) and  +(0,.2) .. (B12);
\draw[black] (B12) .. controls +(.1,.45) and +(-.1,.45) .. (B7);
%%%%%%%%%%%%%%%%%%%%%%%%%%%%%%%
\draw[black] (T4) .. controls +(.1,-.2) and +(-.1,.2) .. (B1);
\draw[black] (T12) .. controls +(.1,-.2) and +(-.1,.2) .. (B9);
\draw[black] (T4) .. controls +(.1,-.2) and +(-.1,.2) .. (B9) ;
%%%%%%%%%%%%%%%%%%%%%%%%%%%%
\draw[black] (B2) .. controls +(.1,.2) and +(-.1,.2) .. (B3);
\draw[black] (B3) .. controls +(.1,.3) and +(-.1,.3) .. (B5);
\draw[black] (B5) .. controls +(.1,.2) and +(-.1,.2) .. (B6);
\draw[black] (B8) .. controls +(.1,.3) and +(-.1,.3) .. (B10) ;
\draw[black] (B6) .. controls +(.1,.3) and +(-.1,.3) .. (B8) ;
\draw[black] (B10) .. controls +(.1,.2) and +(-.1,.2) .. (B11) ;
\draw[black] (T2) .. controls +(0,-.2) and +(0,.2) .. (B2) ;
\draw[black] (T8) .. controls +(0,-.2) and +(0,.2) .. (B8) ;
%%%%%%%%%%%%%%%%%%%%%%%%%%%%%
\draw  (B1)  node[black,below=0.2cm]{$\boldsymbol{3}$};
\draw  (B2)  node[black,below=0.2cm]{$\boldsymbol{4}$};
\draw  (B3)  node[black,below=0.2cm]{$\boldsymbol{4}$};
\draw  (B4)  node[black,below=0.2cm]{$\boldsymbol{1}$};
\draw  (B5)  node[black,below=0.2cm]{$\boldsymbol{4}$};
\draw  (B6)  node[black,below=0.2cm]{$\boldsymbol{4}$};
\draw  (B7)  node[black,below=0.2cm]{$\boldsymbol{2}$};
\draw  (B8)  node[black,below=0.2cm]{$\boldsymbol{4}$};
\draw  (B9)  node[black,below=0.2cm]{$\boldsymbol{3}$};
\draw  (B10)  node[black,below=0.2cm]{$\boldsymbol{4}$};
\draw  (B11)  node[black,below=0.2cm]{$\boldsymbol{4}$};
\draw  (B12)  node[black,below=0.2cm]{$\boldsymbol{2}$};
\draw  (T1)  node[black,above=0.2cm]{$\boldsymbol{2}$};
\draw  (T2)  node[black,above=0.2cm]{$\boldsymbol{4}$};
\draw  (T3)  node[black,above=0.2cm]{$\boldsymbol{1}$};
\draw  (T4)  node[black,above=0.2cm]{$\boldsymbol{3}$};
\draw  (T5)  node[black,above=0.2cm]{$\boldsymbol{1}$};
\draw  (T6)  node[black,above=0.2cm]{$\boldsymbol{2}$};
\draw  (T7)  node[black,above=0.2cm]{$\boldsymbol{2}$};
\draw  (T8)  node[black,above=0.2cm]{$\boldsymbol{4}$};
\draw  (T9)  node[black,above=0.2cm]{$\boldsymbol{1}$};
\draw  (T10)  node[black,above=0.2cm]{$\boldsymbol{2}$};
\draw  (T11)  node[black,above=0.2cm]{$\boldsymbol{2}$};
\draw  (T12)  node[black,above=0.2cm]{$\boldsymbol{3}$};
\draw  (T0) node[black,below=0.35cm]{${\EE_\gamma^\beta = }$};
\foreach \i in {1,...,12}
{ \fill (T\i) circle (2pt); \fill (B\i) circle (2pt); }
\end{tikzpicture}}}.
\end{equation}
Observe that in this example $(\#\gamma_i=j) \equiv  (\#\beta_i = j) \modd r_j$ for
$r_1 = 2, r_2 = 3, r_3 = 2, r_4 = 5$.    Thus, $\EE_{\gamma}^\beta$ is a legitimate
basis element for $\Zs_{12}(\GG)$,  where $\GG = \ZZ_2 \times \ZZ_3 \times \ZZ_2 \times \ZZ_5$.
Since  $\EE_\vartheta^\eta \EE_\gamma^\beta = \delta_{\beta,\vartheta} \EE_{\gamma}^\eta$, the top row of $\EE_\gamma^\beta$ must exactly match the bottom row of $\EE_{\vartheta}^\eta$ to achieve a nonzero product.     Thus for $\EE_\beta^\eta$ with
$\eta = (2,3,2,1,4,2,4,2,3,3,2,3)$, we place the diagram for  $\EE_\beta^\eta$ on top of
the diagram for $\EE_\gamma^\beta$ and concatenate the two diagrams, as pictured
below.

\begin{equation}\label{eq:diag2}
\vcenter{\hbox{%\vspace{1cm}
\begin{tikzpicture}[>=latex,text height=.5ex,text depth=0.1ex, scale=1.25]
\foreach \i in {0,...,12}
{\path(\i,2) coordinate (A\i);  \path (\i,1) coordinate (T\i); \path (\i,0) coordinate (B\i); }
%%%%%%%
\filldraw[fill= gray!50,draw=gray!50,line width=8pt]  (A1) -- (A12) -- (B12) -- (B1) -- (A1);
%%%%%%%%%%%%%%%%%%%%%%%%%%%%
\draw[black] (T3) .. controls +(.1,-.25) and +(.1,-.25) .. (T5) ;
\draw[black] (T5) .. controls +(.1,-.5) and +(-.1,-.5) .. (T9);
\draw[black] (T3) .. controls +(0,-.25) and +(0,.25) .. (B4) ;
%%%%%%%%%%%%%%%%%%%%%%%%%%%%%
\draw[black] (T1) .. controls +(.1,-.48) and +(-.1,-.48) .. (T6);
\draw[black] (T6) .. controls +(.1,-.2) and +(-.1,-.2) .. (T7);
\draw[black] (T7) .. controls +(.1,-.35) and +(-.1,-.35) .. (T10);
\draw[black] (T10) .. controls +(.1,-.2) and +(-.1,-.2) .. (T11);
\draw[black] (T11) .. controls +(0,-.2) and  +(0,.2) .. (B12);
\draw[black] (B12) .. controls +(.1,.45) and +(-.1,.45) .. (B7);
%%%%%%%%%%%%%%%%%%%%%%%%%%%%%%%
\draw[black] (T4) .. controls +(.1,-.2) and +(-.1,.2) .. (B1);
\draw[black] (T12) .. controls +(.1,-.2) and +(-.1,.2) .. (B9);
\draw[black] (T4) .. controls +(.1,-.2) and +(-.1,.2) .. (B9) ;
%%%%%%%%%%%%%%%%%%%%%%%%%%%%
\draw[black] (B2) .. controls +(.1,.2) and +(-.1,.2) .. (B3);
\draw[black] (B3) .. controls +(.1,.3) and +(-.1,.3) .. (B5);
\draw[black] (B5) .. controls +(.1,.2) and +(-.1,.2) .. (B6);
\draw[black] (B8) .. controls +(.1,.3) and +(-.1,.3) .. (B10) ;
\draw[black] (B6) .. controls +(.1,.3) and +(-.1,.3) .. (B8) ;
\draw[black] (B10) .. controls +(.1,.2) and +(-.1,.2) .. (B11) ;
\draw[black] (T2) .. controls +(0,-.2) and +(0,.2) .. (B2) ;
\draw[black] (T8) .. controls +(0,-.2) and +(0,.2) .. (B8) ;
%%%%%%%%%%%%%%%%%%%%%%%%%%%
\draw[black] (A1) .. controls +(0,-.2) and +(0,.2) .. (T1) ;
\draw[black] (A1) .. controls +(.1,-.25) and +(-.1,-.25) .. (A3);
\draw[black] (A3) .. controls +(.1,-.35) and +(-.1,-.35) .. (A6);
\draw[black] (A6) .. controls +(.1,-.3) and +(-.1,-.3) .. (A8);
\draw[black] (A8) .. controls +(.1,-.35) and +(-.1,-.35) .. (A11);
\draw[black] (A11) .. controls +(0,-.2) and  +(0,.2) .. (T11);
\draw[black] (T11) .. controls +(.1,.2) and +(-.1,.2) .. (T10);
\draw[black] (T10) .. controls +(.1,.35) and +(-.1,.35) .. (T7);
\draw[black] (T7) .. controls +(.1,.2) and +(-.1,.2) .. (T6);
%%%%%%%%%%%%%%%%%%%%%%%%%%%%%
\draw[black] (T5) .. controls +(.1,.3) and +(-.1,.3) .. (T3);
\draw[black] (A4) .. controls +(0,-.2) and  +(0,.2) .. (T3);
%%%%%%%%%%%%%%%%%%%%%%%%%%%%
\draw[black] (A9) .. controls +(.1,-.2) and +(-.1,-.2) .. (A10);
\draw[black] (A10) .. controls +(.1,-.35) and +(-.1,-.35) .. (A12);
\draw[black] (A12) .. controls +(0,-.2) and +(0,.2) .. (T12);
\draw[black] (A2) .. controls +(0,-.2) and +(0,.2) .. (T4);
\draw[black] (A9) .. controls +(0,-.2) and  +(0,.2) .. (T4);
%%%%%%%%%%%%%%%%%%%%%%%%%%%%%
\draw[black] (A5) .. controls +(0,-.2) and +(0,.2) .. (T2);
\draw[black] (A5) .. controls +(.1,-.3) and +(-.1,-.3) .. (A7);
\draw[black] (A7) .. controls +(0,-.2) and +(0,.2) .. (T8);
%%%%%%%%%%%%%%%%%%%%%%%%%%%%%
\draw  (B1)  node[black,below=0.2cm]{$\boldsymbol{3}$};
\draw  (B2)  node[black,below=0.2cm]{$\boldsymbol{4}$};
\draw  (B3)  node[black,below=0.2cm]{$\boldsymbol{4}$};
\draw  (B4)  node[black,below=0.2cm]{$\boldsymbol{1}$};
\draw  (B5)  node[black,below=0.2cm]{$\boldsymbol{4}$};
\draw  (B6)  node[black,below=0.2cm]{$\boldsymbol{4}$};
\draw  (B7)  node[black,below=0.2cm]{$\boldsymbol{2}$};
\draw  (B8)  node[black,below=0.2cm]{$\boldsymbol{4}$};
\draw  (B9)  node[black,below=0.2cm]{$\boldsymbol{3}$};
\draw  (B10)  node[black,below=0.2cm]{$\boldsymbol{4}$};
\draw  (B11)  node[black,below=0.2cm]{$\boldsymbol{4}$};
\draw  (B12)  node[black,below=0.2cm]{$\boldsymbol{2}$};
\draw  (A1)  node[black,above=0.2cm]{$\boldsymbol{2}$};
\draw  (A2)  node[black,above=0.2cm]{$\boldsymbol{3}$};
\draw  (A3)  node[black,above=0.2cm]{$\boldsymbol{2}$};
\draw  (A4)  node[black,above=0.2cm]{$\boldsymbol{1}$};
\draw  (A5)  node[black,above=0.2cm]{$\boldsymbol{4}$};
\draw  (A6)  node[black,above=0.2cm]{$\boldsymbol{2}$};
\draw  (A7)  node[black,above=0.2cm]{$\boldsymbol{4}$};
\draw  (A8)  node[black,above=0.2cm]{$\boldsymbol{2}$};
\draw  (A9)  node[black,above=0.2cm]{$\boldsymbol{3}$};
\draw  (A10)  node[black,above=0.2cm]{$\boldsymbol{3}$};
\draw  (A11)  node[black,above=0.2cm]{$\boldsymbol{2}$};
\draw  (A12)  node[black,above=0.2cm]{$\boldsymbol{3}$};
\draw  (T1)  node[black,right]{$\boldsymbol{2}$};
\draw  (T2)  node[black,right]{$\boldsymbol{4}$};
\draw  (T3)  node[black,right]{$\boldsymbol{1}$};
\draw  (T4)  node[black,right]{$\boldsymbol{3}$};
\draw  (T5)  node[black,right]{$\boldsymbol{1}$};
\draw  (T6)  node[black,right]{$\boldsymbol{2}$};
\draw  (T7)  node[black,right]{$\boldsymbol{2}$};
\draw  (T8)  node[black,right]{$\boldsymbol{4}$};
\draw  (T9)  node[black,right]{$\boldsymbol{1}$};
\draw  (T10)  node[black,right]{$\boldsymbol{2}$};
\draw  (T11)  node[black,right]{$\boldsymbol{2}$};
\draw  (T12)  node[black,right]{$\boldsymbol{3}$};
\draw  (T0) node[black,below=0.35cm]{${\EE_\gamma^\beta = }$};
\draw  (A0) node[black,below=0.35cm]{${\EE_\beta^\eta = }$};
\foreach \i in {1,...,12}
{ \fill (A\i) circle (2pt); \fill (B\i) circle (2pt);  \fill (T\i) circle (2pt);}
\end{tikzpicture}}}
\end{equation}

The result is $\EE_{\beta}^\eta \EE_{\gamma}^\beta =  \EE_{\gamma}^\eta$ where

\begin{equation}\label{eq:diag3}
\vcenter{\hbox{%\vspace{1cm}
\begin{tikzpicture}[>=latex,text height=.5ex,text depth=0.1ex, scale=1.25]
\foreach \i in {0,...,12}
{ \path (\i,1) coordinate (T\i); \path (\i,0) coordinate (B\i); }
%%%%%%%
\filldraw[fill= gray!50,draw=gray!50,line width=8pt]  (T1) -- (T12) -- (B12) -- (B1) -- (T1);
%%%%%%%%%%%%%%%%%%%%%%%%%%%%
\draw[black] (T4) .. controls +(0,-.2) and +(0,.2) .. (B4);
%%%%%%%%%%%%%%%%%%%%%%%%%%%%%
\draw[black] (T1) .. controls +(.1,-.3) and +(-.1,-.3) .. (T3);
\draw[black] (T3) .. controls +(.1,-.35) and +(-.1,-.35) .. (T6);
\draw[black] (T6) .. controls +(.1,-.3) and +(-.1,-.3) .. (T8);
\draw[black] (T8) .. controls +(.1,-.35) and +(-.1,-.35) .. (T11);
\draw[black] (T11) .. controls +(0,-.2) and  +(0,.2) .. (B12);
\draw[black] (B12) .. controls +(.1,.45) and +(-.1,.45) .. (B7);
%%%%%%%%%%%%%%%%%%%%%%%%%%%%%
\draw[black] (T2) .. controls +(0,-.2) and  +(0,.2) .. (B1);
\draw[black] (T2) .. controls +(.1,-.55) and +(-.1,-.55) .. (T9);
\draw[black] (T9) .. controls +(0,-.2) and  +(0,.2) .. (B9);
\draw[black] (T9) .. controls +(.1,-.2) and +(-.1,-.2) .. (T10);
\draw[black] (T10) .. controls +(.1,-.3) and +(-.1,-.3) .. (T12);
%%%%%%%%%%%%%%%%%%%%%%%%%%%%
\draw[black] (T7) .. controls +(0,-.2) and  +(0,.2) .. (B8);
\draw[black] (T5) .. controls +(.1,-.35) and +(-.1,-.35) .. (T7);
\draw[black] (B2) .. controls +(.1,.2) and +(-.1,.2) .. (B3);
\draw[black] (B3) .. controls +(.1,.3) and +(-.1,.3) .. (B5);
\draw[black] (B5) .. controls +(.1,.2) and +(-.1,.2) .. (B6);
\draw[black] (B8) .. controls +(.1,.3) and +(-.1,.3) .. (B10) ;
\draw[black] (B6) .. controls +(.1,.3) and +(-.1,.3) .. (B8) ;
\draw[black] (B10) .. controls +(.1,.2) and +(-.1,.2) .. (B11) ;
%\draw[black] (T2) .. controls +(0,-.2) and +(0,.2) .. (B2) ;
%%%%%%%%%%%%%%%%%%%%%%%%%%%%%
\draw  (B1)  node[black,below=0.2cm]{$\boldsymbol{3}$};
\draw  (B2)  node[black,below=0.2cm]{$\boldsymbol{4}$};
\draw  (B3)  node[black,below=0.2cm]{$\boldsymbol{4}$};
\draw  (B4)  node[black,below=0.2cm]{$\boldsymbol{1}$};
\draw  (B5)  node[black,below=0.2cm]{$\boldsymbol{4}$};
\draw  (B6)  node[black,below=0.2cm]{$\boldsymbol{4}$};
\draw  (B7)  node[black,below=0.2cm]{$\boldsymbol{2}$};
\draw  (B8)  node[black,below=0.2cm]{$\boldsymbol{4}$};
\draw  (B9)  node[black,below=0.2cm]{$\boldsymbol{3}$};
\draw  (B10)  node[black,below=0.2cm]{$\boldsymbol{4}$};
\draw  (B11)  node[black,below=0.2cm]{$\boldsymbol{4}$};
\draw  (B12)  node[black,below=0.2cm]{$\boldsymbol{2}$};
\draw  (T1)  node[black,above=0.2cm]{$\boldsymbol{2}$};
\draw  (T2)  node[black,above=0.2cm]{$\boldsymbol{3}$};
\draw  (T3)  node[black,above=0.2cm]{$\boldsymbol{2}$};
\draw  (T4)  node[black,above=0.2cm]{$\boldsymbol{1}$};
\draw  (T5)  node[black,above=0.2cm]{$\boldsymbol{4}$};
\draw  (T6)  node[black,above=0.2cm]{$\boldsymbol{2}$};
\draw  (T7)  node[black,above=0.2cm]{$\boldsymbol{4}$};
\draw  (T8)  node[black,above=0.2cm]{$\boldsymbol{2}$};
\draw  (T9)  node[black,above=0.2cm]{$\boldsymbol{3}$};
\draw  (T10)  node[black,above=0.2cm]{$\boldsymbol{3}$};
\draw  (T11)  node[black,above=0.2cm]{$\boldsymbol{2}$};
\draw  (T12)  node[black,above=0.2cm]{$\boldsymbol{3}$};
\draw  (T0) node[black,below=0.35cm]{${\EE_\gamma^\eta = }$};
\foreach \i in {1,...,12}
{ \fill (T\i) circle (2pt); \fill (B\i) circle (2pt); }
\end{tikzpicture}}}.
\end{equation}
\bigskip \medskip

\section{Appendix I}\label{S:graphwalk}
Let $\mathcal{G}$ be a directed graph with finite vertex set $\Gamma$ and adjacency
matrix $\A = (a_{\alpha,\gamma})_{\alpha,\gamma \in \Gamma}$.    Then  $a_{\alpha,\gamma}$ is
the number of edges (arrows) from $\alpha$ to $\gamma$ in $\mathcal G$, and $(\A^k)_{\alpha,\gamma}$ is
the number of walks of $k$ steps from $\alpha$ to $\gamma$ on $\mathcal{G}$.   We consider the corresponding generating function for the number of walks from $\alpha$ to $\gamma$,

$$\ws_{\alpha,\gamma}(t) =  \sum_{k = 0}^\infty  \, (\A^k)_{\alpha,\gamma}\, t^k,$$
where $\A^0 = \mathrm{I}$, the identity matrix.
\medskip

\begin{prop}   Let $\delta_{\alpha}$  be the $|\Gamma|\times 1$ matrix with 1 in row $\alpha$
and zeros elsewhere so that  entry $\gamma$ of $\delta_{\alpha}$ is the Kronecker delta $\delta_{\alpha,\gamma}$,  and assume $\col_\alpha^\gamma$ is the matrix  $\mathrm{I}-t \A^{\tt T}$ with column $\gamma$ replaced by $\delta_\alpha$  (here ${\tt T}$ denotes the transpose).     Then
$$\ws_{\alpha,\gamma}(t)  = \frac{\mathsf{det}(\col_\alpha^\gamma)}{\det(\mathrm{I}-t \A)}.$$
\end{prop}

\begin{proof}   First a simple observation:
$(\A^{k+1})_{\alpha,\gamma} = \sum_{\beta \in \Gamma}  (\A^k)_{\alpha,\beta}\, a_{\beta,\gamma},$
for all $k \geq 0$.     Then

\begin{align*} \ws_{\alpha,\gamma}(t) &= \sum_{k = 0}^\infty  \, (\A^k)_{\alpha,\gamma}\, t^k \\
&= \delta_{\alpha,\gamma} + t \sum_{k \geq 1} \  (\A^k)_{\alpha,\gamma}\, t^{k-1}  \\
&=  \delta_{\alpha,\gamma}+ t \sum_{k \geq 0} \  (\A^{k+1})_{\alpha,\gamma}\, t^{k}\\
&= \delta_{\alpha,\gamma} + t \sum_{k\geq 0} \left(\sum_{\beta \in \Gamma}  (\A^k)_{\alpha,\beta}\, a_{\beta,\gamma}  \right) t^{k} \\
&= \delta_{\alpha.\gamma} + t \sum_{\beta \in \Gamma} a_{\beta,\gamma} \left( \sum_{k\geq 0}    (\A^k)_{\alpha,\beta} \,t^k \right) \\
&=\delta_{\alpha,\gamma} + t \,\sum_{\beta \in \Gamma} \, a_{\beta, \gamma}\,\ws_{\alpha,\beta}(t).
\end{align*}

Letting $\ws_\alpha$  be the $|\Gamma|\times 1$ matrix with $\ws_{\alpha,\gamma}(t)$ in row $\gamma$, we see from the above calculation that  the matrix equation
$\ws_\alpha^{\mathsf{T}}\left(\mathrm{I} - t \A\right) = \delta_\alpha^{\mathsf T}$,  or equivalently, $\left(\mathrm{I} - t \A^{\tt T}\right)\ws_\alpha = \delta_\alpha$ holds.   It follows then from Cramer's rule that
$$\ws_{\alpha,\gamma}(t) =  \frac{ \mathsf{det}(\col_{\alpha}^\gamma)}  {\mathsf{det}(\mathrm{I}-t\A^{\tt T})}
=  \frac{ \mathsf{det}(\col_{\alpha}^\gamma)}{\mathsf{det}(\mathrm{I}-t\A)}.$$
\end{proof}

\section{Appendix II}\label{S:Bdiag}
\smallskip
Levels 0-6 of the Bratteli diagram for the cyclic group $\GG =\ZZ_{10}$ and its module  $\VV =\GG_{1} \oplus \GG_{9}$ are pictured below. The label inside the node is the index of the irreducible $\GG$-module.  The trivial module is indicated in white,  and the module
$\VV$ in black.   The subscript  on node $\lam$ on level $k$ indicates
the number of paths from $\zr$ at the top to $\lam$ at level $k$ (equivalently, the number of walks from $\zr$ to $\lam$
of $k$ steps on the McKay quiver $\cR_{\VV}(\GG)$; also the multiplicity of $\GG_\lam$ in $\VV^{\ot k}$;
also the dimension of the irreducible module $\Zs_k^\lam(\GG)$ for the centralizer algebra
$\Zs_k(\GG) = \End_{\GG}(\VV^{\ot k})$).

\[\includegraphics[scale=1,page=2]{group-walk-diagrams.pdf}\]

\medskip

 \noindent \textit{\small Department of Mathematics, University of Wisconsin-Madison, Madison, WI 53706, USA}\\
{\small benkart@math.wisc.edu}

\noindent \textit{\small Department of Mathematics, Sejong University,
Seoul, 133-747,  Korea  (ROK)}\\
{\small dhmoon@sejong.ac.kr}

\end{document}